\documentclass[11pt,a4paper,english]{article}

\newcommand{\beq}{\begin{equation}}
\newcommand{\eeq}{\end{equation}}

\usepackage[latin1]{inputenc}
\usepackage{babel}
\usepackage[centertags]{amsmath}
\usepackage{amsfonts}
\usepackage{amssymb}
\usepackage{color}
\usepackage{amsthm}
\usepackage{newlfont}
\usepackage{graphicx}
\usepackage{dsfont}
\usepackage{geometry}
\usepackage{mathrsfs}

\newtheorem{theorem}{Theorem}[section]
\newtheorem{lemma}[theorem]{Lemma}
\newtheorem{coroll}[theorem]{Corollary}
\newtheorem{prop}[theorem]{Proposition}
\newtheorem{definition}[theorem]{Definition}
\newtheorem{remark}[theorem]{Remark}
\newtheorem{ass}[theorem]{Assumption}
\newtheorem{example}[theorem]{Example}
\newcommand*{\lnr}{\left|\mkern-1mu\left|\mkern-1mu\left|}
\newcommand*{\rnr}{\right|\mkern-1mu\right|\mkern-1mu\right|}
\newcommand{\msc}[1]{\textbf{MSC2010 Classification:} #1.}
\newcommand{\keywords}[1]{\textbf{Key words:} #1.}
\newcommand{\ackn}[1]{\textbf{Acknowledgments:} #1.}

\makeatletter  
\@addtoreset{equation}{section}
\def\theequation{\arabic{section}.\arabic{equation}}
\makeatother  
\begin{document}
\title{\textbf{Optimal stopping of a Hilbert space valued diffusion:\\ an infinite dimensional variational inequality}\thanks{These results extend a portion of the second Author PhD dissertation \cite{PhD-T} under the supervision of the first Author. Both Authors wish to thank Franco Flandoli  and Claudio Saccon for their helpful comments and suggestions.}}
\author{Maria B. Chiarolla\thanks{Dipartimento di Metodi e Modelli per l'Economia, il Territorio e la Finanza (MEMOTEF), Universit\`a di Roma `La Sapienza', via del Castro Laurenziano 9, 00161 Roma, Italy; \texttt{maria.chiarolla@uniroma1.it}}\:\:\:\:\:\:and\:\:\:\:\:Tiziano De Angelis\thanks{Corresponding author.~School of Mathematics, University of Manchester, Oxford Rd.~M13 9PL Manchester (UK); \texttt{tiziano.deangelis@manchester.ac.uk}
}
}
\maketitle
\begin{abstract}
A finite horizon optimal stopping problem for an infinite dimensional diffusion $X$ is analyzed by means of variational techniques. The diffusion is driven by a SDE on a Hilbert space $\mathcal{H}$ with a non-linear diffusion coefficient $\sigma(X)$ and a generic unbounded operator $A$ in the drift term. When the gain function $\Theta$ is time-dependent and fulfils mild regularity assumptions, the value function $\mathcal{U}$ of the optimal stopping problem is shown to solve an infinite-dimensional, parabolic, degenerate variational inequality on an unbounded domain. Once the coefficient $\sigma(X)$ is specified, the solution of the variational problem is found in a suitable Banach space $\mathcal{V}$ fully characterized in terms of a Gaussian measure $\mu$.

This work provides the infinite-dimensional counterpart, in the spirit of Bensoussan and Lions \cite{Ben-Lio82}, of well-known results on optimal stopping theory and variational inequalities in $\mathbb{R}^n$. These results may be useful in several fields, as in mathematical finance when pricing American options in the HJM model.
\end{abstract}
\msc{60G40, 49J40, 35R15}
\vspace{+8pt}

\noindent\keywords{optimal stopping, infinite-dimensional stochastic analysis, parabolic partial dif\-fe\-ren\-tial equations, degenerate variational inequalities}

\section{Introduction}

This paper studies a finite horizon optimal stopping problem associated to an infinite-dimensional diffusion process by means of variational techniques. It is well known that the value function of a wide class of optimal stopping problems for general diffusions in $\mathbb{R}^n$ may be characterized as the solution of suitable variational problems (see \cite{Ben-Lio82} and references therein for a survey). Here we provide an infinite-dimensional counterpart of those results by extending methods employed in \cite{Ben-Lio82} and combining them with techniques borrowed from the theory of infinite dimensional SDEs.

This work is partially motivated by a central problem in the modern theory of mathematical finance. In fact, pricing American bond options on the forward interest rate curve gives rise to an infinite dimensional optimal stopping problem. This is a consequence of the dependence of the bond's price on the whole structure of the forward curve. The results obtained here will be extended to solve that particular financial problem in a forthcoming paper \cite{Ch-DeA12b}.

Optimal stopping for processes in locally compact spaces has attracted great attention in the last decades (cf.~\cite{ElKar81}, \cite{Shir}, \cite{Zab84} among others) while the case of general infinite-dimensional Markov processes has been studied in relatively few papers. The earliest paper on infinite dimensional optimal stopping and variational inequalities we are aware of is \cite{Ch-Men89}. There Chow and Menaldi extended known finite dimensional results, in the spirit of \cite{Ben-Lio82}, to the case of a particular infinite dimensional linear diffusion.

A first attempt towards a more comprehensive study of optimal stopping theory for processes taking values in a Polish space was made by J. Zabczyk  \cite{Zab97} in 1997 from a purely probabilistic point of view and later on, in 2001, by variational methods \cite{Zab01}. Recently Barbu and Marinelli \cite{Bar-Mar08} contributed further insights in this direction adopting arguments similar to those in Zabczyk's works. In both \cite{Bar-Mar08} and \cite{Zab01} the Authors considered a diffusion process on a functional space $\mathcal{H}$ and solved the variational problem in \emph{mild sense} in a suitable $L^2$-space with respect to a measure on $\mathcal{H}$. Instead in the present work we find Sobolev-type solutions (therefore local) of the variational problem. 
Barbu and Sritharan \cite{Bar-Sri06} also considered an optimal stopping problem for a 2-dimensional stochastic Navier-Stokes equation and solved the associated infinite dimensional variational inequality in a $L^2$-space.

A different approach is based on viscosity theory. It is extensively exploited to solve general stochastic control problems (cf.~\cite{Fl-Son} for a survey) and the infinite-dimensional case is currently the object of intense study (cf.~\cite{Lions89-1}, \cite{Lions89-2}, \cite{Lions89-3}, \cite{Sw94} among others). However, as far as we know, the only paper on infinite-dimensional variational inequalities related to optimal stopping problems studied by viscosity methods is \cite{Gat-Sw99} by D.\ G\c{a}tarek and A.\ \'Swi\c{e}ch. The Authors deal with a problem arising in finance. They characterize the value function of the optimal stopping problem when the underlying diffusion has a particular form not involving the unbounded term that normally appears in infinite-dimensional stochastic differential equations (cf.~\cite{DaPr-Zab} for a survey).

It is worth mentioning that attempts to provide some numerical results for this class of problems were  recently made in 
\cite{Marc08}. However, arguments therein are mostly heuristic, proofs are only sketched and some of them seem incorrect.

In the present paper the underlying process $X$ lives in a general Hilbert space $\mathcal{H}$ and it is governed by the SDE \eqref{SDE-infty} with a generic unbounded operator $A$ (which is not even required to be self-adjoint) and with diffusion coefficient $\sigma$ in a class of functions which depends on $A$ through Assumptions \ref{ass-Q} and \ref{ass-sigma} (see below the discussion after Remark \ref{suffregtheta}). Under mild regularity assumptions on the gain function $\Theta$, the value function $\mathcal{U}$ of the corresponding optimal stopping problem (see \eqref{OS1}) solves an infinite dimensional variational inequality that is parabolic and highly degenerate on an unbounded domain. We point out that degenerate variational inequalities represent non-standard problems in the context of PDE theory even at the finite dimensional level (cf.\ \cite{Str-Var72}). For the associated optimal stopping problems one may consult the work of J.L.\ Menaldi \cite{Men80}, \cite{Men80bis}. In our case we show that $\mathcal{U}$ solves a variational inequality in a specific Sobolev-type space $\mathcal{V}$ (cf.\ \eqref{def-sp-V}) under a given centered Gaussian measure $\mu$ (cf.\ \eqref{gaussmeas}). We also obtain uniqueness at least in a special case under more restrictive assumptions on $X$ (see Section \ref{uniqueness}).

This work is ideally the extension of \cite{Ch-Men89} to general diffusions in Hilbert spaces and it provides the infinite dimensional analogue of the results in \cite{Men80}, \cite{Men80bis}. {{Differently to \cite{Ch-Men89} we consider a finite time-horizon and a SDE with a generic non-linear diffusion coefficient. The problem in \cite{Ch-Men89} is analyzed as a special case of our study and two open questions raised in \cite{Ch-Men89} find positive answers in our Section \ref{uniqueness}.}}

The paper is organized as follows. In Section \ref{setting} we set the problem and we make the main regularity assumptions on the diffusion $X$ and on the gain function $\Theta$. Then we obtain regularity of the value function $\mathcal{U}$. Section \ref{sec-appr} deals with the approximation of the SDE \eqref{SDE-infty} and of the optimal stopping problem \eqref{OS1}. The SDE is approximated in two steps: first the unbounded term $A$ is replaced by its Yosida approximation $A_\alpha$, $\alpha>0$, and afterwards a $n$-dimensional reduction of the SDE is obtained. In this approximation procedure the corresponding process $X^{(\alpha);n}$ gives rise to an optimal stopping problem whose value function we denote by $\mathcal{U}^{(n)}_\alpha$. By means of purely probabilistic arguments we show that $\mathcal{U}^{(n)}_\alpha$ converges to the value function $\mathcal{U}$ of the original optimal stopping problem for $n\to\infty$ and $\alpha\to\infty$. The variational problem is studied in Section \ref{sec-fdvi}. Initially we prove that the value function $\mathcal{U}_\alpha^{(n)}$ is solution of a suitable variational inequality in $\mathbb{R}^n$ and we characterize an optimal stopping time. We also provide a number of important bounds on $\mathcal{U}_\alpha^{(n)}$, its time derivative and its gradient, by means of penalization methods. Section \ref{sec-infty} is entirely devoted to prove that our original value function $\mathcal{U}$ solves a suitable infinite-dimensional variational problem. The result is obtained by taking the limit as $n\to\infty$ and $\alpha\to\infty$ of the variational problem detailed in Sections \ref{sec:fd01} and \ref{sec:fd02}. Both analytical and probabilistic tools are adopted to carry out the proofs and to characterize an optimal stopping time. In Section \ref{uniqueness} uniqueness of the solution to the variational problem is proved for a specific class of diffusion processes.
The paper is completed by a technical Appendix containing some proofs.


\section{Setting and preliminary estimates}\label{setting}
Let $\mathcal{H}$ be a separable Hilbert space with scalar product $\langle\,\cdot\,,\,\cdot\,\rangle_{\mathcal{H}}$ and induced norm $\|\cdot\|_{\mathcal{H}}$. Let  $A:D(A)\subset\mathcal{H}\to\mathcal{H}$ be the infinitesimal generator of a strongly continuous semigroup of operators $\{S(t),t\geq0\}$ on $\mathcal{H}$ (cf.\ \cite{Pazy}), where $D(A)$ denotes its domain. Notice that $D(A)$ is dense in $\mathcal{H}$. Let $\{\varphi_1,\varphi_2,\ldots\}$ be an orthonormal basis of $\mathcal{H}$ with $\varphi_i\in D(A)$, $i=1,2,\ldots$.

We now consider a stochastic framework. Let $(\Omega,\mathcal{F},\mathbb{P})$ be a complete probability space and let $W:=(W^0,W^1,W^2,\ldots)$ be a sequence of independent, real, standard Brownian motions on it. The filtration generated by $W$ is $\{\mathcal{F}_t,t\geq0\}$ and it is completed by the null sets. Fix a finite horizon $T>0$ and take a continuous map $\sigma:\mathcal{H}\to\mathcal{H}$ whose regularity will be specified later in this section (cf. Assumption \ref{ass-sigma}). Consider the stochastic differential equation (SDE)
\begin{equation}\label{SDE-infty}
\left\{
\begin{array}{l}
dX_t=AX_tdt+\sigma(X_t)dW^0_t,\qquad t\in[0,T],\\
\\
X_0=x,
\end{array}
\right.
\end{equation}
in $\mathcal{H}$. We denote by $X^x$ a mild solution of \eqref{SDE-infty}. When the starting time is $t$ rather than zero the solution is denoted by $X^{t,x}$.
To simplify exposition we have chosen an SDE driven by a 1-dimensional Brownian motion, however our results may be also extended to $\mathcal{H}$-valued Brownian motions with trace-class covariance operator. In this paper we will rely on the infinite sequence of Brownian motions $W$ to find finite dimensional approximations of $X^x$ driven by a SDEs similar to \eqref{SDE-infty} but with Brownian motion $\overline{W}^{(n)}:=(W^0,\ldots,W^n)^\intercal$ instead of $W^0$.

We aim to study the infinite dimensional optimal stopping problem
\begin{equation}\label{OS1}
\mathcal{U}(t,x):=\sup_{t\leq\tau\leq T}\mathbb{E}\left\{\Theta(\tau,X^{t,x}_\tau)\right\},
\end{equation}
with $\tau$ a stopping time with respect to the filtration $\{\mathcal{F}_t,t\in[0,T]\}$ and with gain function $\Theta:[0,T]\times\mathcal{H}\to\mathbb{R}$ such that $\Theta\ge0$ and $(t,x)\mapsto\Theta(t,x)$ continuous. Although infinite-dimensional optimal stopping problems like \eqref{OS1} have been proposed by several Authors (see, for example, \cite{Bar-Mar08}, \cite{Ch-Men89}, \cite{Gat-Sw99}, \cite{Zab97}, \cite{Zab01}), here we provide an alternative method to characterize 
the value function $\mathcal{U}$. Our results might be extended to the case of a discounted gain function, if the discount factor is a Lipschitz-continuous, non-negative function of $X$. In order to work out problem \eqref{OS1} we need to specify some properties of $\Theta$, $\sigma$ and $A$. For that we introduce suitable Gauss-Sobolev spaces.

Define the positive, linear operator $Q:\mathcal{H}\to\mathcal{H}$ by
\begin{equation}\label{bari}
Q\varphi_i=\lambda_i\varphi_i, \qquad\lambda_i>0,\qquad i=1,2,\ldots,
\end{equation}
with $\sum^\infty_{i=1}{\lambda_i}<\infty$; i.e., $Q$ is of trace class. Consider the centered Gaussian measure $\mu$ with covariance operator $Q$ (cf. \cite{Bogachev}, \cite{DaPr}, \cite{DaPr-Zab04}); that is, the restriction to the vectors\footnote{$\ell_2$ denotes the set of infinite vectors $x:=(x_1,x_2,\ldots)$ such that $\sum_{k}{x_k^2}<+\infty$.} $x\in\ell_2$ of the infinite product measure
\begin{eqnarray}\label{gaussmeas}
\mu(dx)=\prod^\infty_{i=1}{\frac{1}{\sqrt{2\pi\lambda_i}}e^{-\frac{x^2_i}{2\lambda_i}}dx_i}.
\end{eqnarray}
For $1\leq p<+\infty$ and $f:\mathcal{H}\to\mathbb{R}$, define the $L^p(\mathcal{H},\mu)$ norm as
\begin{align}\label{Lp-n1}
&\|f\|_{L^p({\mathcal{H},\mu})}:=\left(\int_{\mathcal{H}}{|f(x)|^p\mu(dx)}\right)^{\frac{1}{p}}.
\end{align}
Then, with the notation of \cite{DaPr}, Chapter 10, we consider derivatives in the Friedrichs sense; that is,
\begin{align}\label{Dfried}
D_k\,f(x):=\lim_{\varepsilon\to 0}\frac{1}{\varepsilon}\left[f(x+\varepsilon\varphi_k)-f(x)\right],\qquad x\in\mathcal{H},\,k\in\mathbb{N},
\end{align}
when the limit exists.

{If $f$ belongs to the domain of the gradient operator $D$ and $\mathcal{H}$ is identified with its dual, then the $L^p(\mathcal{H},\mu;\mathcal{H})$ norm of $Df=\left(D_1f,\,D_2f,\,\ldots\right)$ is defined as
\begin{align}
&\|Df\|_{L^p({\mathcal{H},\mu};\mathcal{H})}:=\left(\int_{\mathcal{H}}{\|Df(x)\|^p_{\mathcal{H}}\,\mu(dx)}\right)
^{\frac{1}{p}}\qquad\text{for $1\le p<+\infty$}\label{Lp-n2}
\end{align}
where
\begin{align}\label{Dfried02}
\big\|Df(x)\big\|_{\mathcal{H}}=\Big(\sum_{k}\big|D_kf(x)\big|^2\Big)^{\frac{1}{2}}<+\infty.
\end{align}}
One can show that $D$ is closable in $L^p(\mathcal{H},\mu)$ (cf.\ \cite{DaPr}, Chapter 10). Let $\overline{D}$ denote the closure of $D$ in $L^p(\mathcal{H},\mu)$ and define the Sobolev space
\begin{eqnarray}\label{def-W1p}
W^{1,p}(\mathcal{H},\mu):=\{f: f\in L^p(\mathcal{H},\mu)\,\text{and}\,\overline{D}f\in L^p(\mathcal{H},\mu;\mathcal{H})\}.
\end{eqnarray}
Notice however that in the case of generalized derivatives $D$ and $\overline{D}$ are the same.

For $n\in\mathbb{N}$ the finite dimensional counterpart of $\mu$, $L^p({\mathcal{H},\mu})$, $L^p(\mathcal{H},\mu;\mathcal{H})$ are, respectively,
\begin{equation*}
\mu_n(dx):=\prod^n_{i=1}{\frac{1}{\sqrt{2\pi\lambda_i}}e^{-\frac{x^2_i}{2\lambda_i}}dx_i},\quad L^p(\mathbb{R}^n,\mu_n),\quad L^p(\mathbb{R}^n,\mu_n;\mathbb{R}^n).
\end{equation*}
\begin{remark}\label{conn}
If $f:\mathbb{R}^n\to\mathbb{R}$, then
\begin{align*}
\|f\|_{L^p({\mathcal{H},\mu})}=\left(\int_{\mathcal{H}}{|f(x)|^p\mu(dx)}\right)^{\frac{1}{p}}=\left(
\int_{\mathbb{R}^n}{|f(x)|^p\mu_n(dx)}\right)^{\frac{1}{p}}=:\|f\|_{L^p({\mathbb{R}^n,\mu_n})}
\end{align*}
and
\begin{align*}
\|Df\|_{L^p({\mathcal{H},\mu};\mathcal{H})}=\left(\int_{\mathcal{H}}{\|Df(x)\|^p_{\mathcal{H}}\,\mu(dx)}\right)
^{\frac{1}{p}}=\left(\int_{\mathbb{R}^n}{\|Df(x)\|^p_{\mathbb{R}^n}\,\mu_n(dx)}\right)^{\frac{1}{p}}=:
\|Df\|_{L^p({\mathbb{R}^n,\mu_n};\mathbb{R}^n)}\,.
\end{align*}
\end{remark}

Again as in \cite{DaPr}, Chapter 10, we define
\begin{align}\label{Dfried03}
D_kD_jf(x):=\lim_{\varepsilon\to0}\frac{1}{\varepsilon}\left[D_jf(x+\varepsilon\varphi_k)-D_jf(x)\right],\qquad x\in\mathcal{H},\,k\in\mathbb{N},
\end{align}
when the limit exists. {For functions $f$ in the domain of $D^2$ one has} $D^2f:\mathcal{H}\to \mathcal{L}(\mathcal{H})$ where $\mathcal{L}(\mathcal{H})$ denotes the space of linear operators on $\mathcal{H}$. In this paper we do not need an $L^p$-space associated to the second derivative.

At this point we can go back to our optimal stopping problem \eqref{OS1} and make the following regularity assumptions on the gain function $\Theta$.
\begin{ass}\label{ass-psi}
There exist positive constants $\overline{\Theta}$, $L_\Theta$, $L^\prime_\Theta$ such that
\begin{align}
&0\le\Theta(t,x)\le\overline{\Theta}\:\:\text{on $[0,T]\times\mathcal{H}$},\label{psi1}\\
\nonumber\\
&(t,x)\mapsto D\Theta(t,x)\:\text{is continuous and}\:\: \|D\Theta(t,x)\|_{\mathcal{H}}\leq L_{\Theta}\qquad t\in[0,T],\,x\in\mathcal{H},\label{psi2}\\
\nonumber\\
&(t,x)\mapsto \frac{\partial\Theta}{\partial t}(t,x)\:\text{is continuous and}\:\: \Big|\frac{\partial\Theta}{\partial t}(t,x)\Big|\leq L^\prime_{\Theta} \qquad t\in[0,T],\,x\in\mathcal{H}.\label{psi3}
\end{align}
Also, $(t,x)\mapsto D^2\Theta(t,x)$ is continuous and $\sup_{(t,x)\in[0,T]\times\mathcal{H}}\big\|D^2\Theta(t,x)\big\|_{L}<+\infty$ with $\|\,\cdot\,\|_L$ the norm in $\mathcal{L}(\mathcal{H})$.
\end{ass}
Obviously Assumption \ref{ass-psi} implies
\begin{equation}
\sup_{t\in[0,T]}\|\Theta(t)\|_{W^{1,p}(\mathcal{H},\mu)}<C_\Theta\quad\text{and}\quad\int_0^T{\bigg\|\frac{\partial\Theta}{\partial t}(t)\bigg\|^2_{L^2(\mathcal{H},\mu)}dt}<C_{\Theta}\,.
\end{equation}
for some positive constant $C_\Theta$ and all $1\le p<+\infty$. In what follows condition \eqref{psi2} will be often referred to as Lipschitz property of the gain function $\Theta$.

\begin{remark}\label{suffregtheta}
Notice that for existence results of the variational problem \eqref{d-2tris} associated to the optimal stopping one \eqref{OS1}, we could assume
\begin{align}\label{condth01}
\left\{
\begin{array}{l}
\vspace{+10pt}
\text{$\Theta\ge0$, $(t,x)\mapsto\Theta(t,x)$ continuous on $[0,T]\times\mathcal{H}$,}\\
\vspace{+10pt}
\text{$\big|\Theta(t,x)\big|\le C\big(1+\|x\|^p_\mathcal{H}\big)$ on $[0,T]\times\mathcal{H}$ for some $1\le p<+\infty$ and $C>0$,}\\
\vspace{+10pt}
\sup_{0\le t\le T}\big|\Theta(t,x)-\Theta(t,y)\big|\le L_\Theta\big\|x-y\big\|_{\mathcal{H}}\quad\text{for $L_\Theta>0$, $x,y\in\mathcal{H}$,} \\
\frac{\partial\Theta}{\partial t}\in L^2(0,T;L^2(\mathcal{H},\mu)).
\end{array}
\right.
\end{align}
However, such $\Theta$ may be approximated by regular ones satisfying Assumption \ref{ass-psi}, for example exponential functions as in \cite{DaPr} or cylindrical ones as in \cite{Bogachev} or \cite{Ma-Rock}.
\end{remark}
The dynamics \eqref{SDE-infty} is fully specified in terms of $A$ and $\sigma$. In applications of infinite dimensional SDEs the choice of the unbounded operator $A$ is often distinctive of the phenomenon that one wants to describe (for example it may involve the Laplacian in Navier-Stokes equations or the first derivative in delay equations), whereas multiple choices of the diffusion coefficient are possible in several situations (see for instance various versions of Musiela's model for interest rates). In our setting we allow for a very general operator $A$ at the cost of restricting the class of admissible diffusion coefficients $\sigma$. In fact, given $A$ and denoted by $A^*$ its adjoint operator we construct $Q$ verifying the following
\begin{ass}\label{ass-Q}
The covariance operator $Q$ of \eqref{bari} is such that
\begin{equation}\label{tecn}
Tr\left[A\,Q\,A^*\right]<\infty.
\end{equation}
\end{ass}
\noindent The above condition is needed in Section \ref{sec-fdvi}, however such $Q$ always exists. Indeed given an orthonormal basis $(\varphi_j)_{j\in\mathbb{N}}\subset D(A)$ of $\mathcal{H}$, the operator $Q$ is constructed by picking its eigenvalues $(\lambda_i)_{i\in\mathbb{N}}$ so that $\sum_{j=1}^\infty\lambda_j\big\|A\varphi_j\big\|^2_\mathcal{H}<+\infty$, which is equivalent to say that \eqref{tecn} holds. Once $Q$ is constructed the class of diffusion coefficients $\sigma$ is determined by
\begin{ass}\label{ass-sigma}
The diffusion coefficient of \eqref{SDE-infty} is such that
\begin{align*}
\left\{
\begin{array}{l}
\text{$(1)$ $\sigma(x)\in Q(\mathcal{H})$, $\forall x\in\mathcal{H}$\, (i.e., $\sigma(x)= Q\gamma(x)$ for some $\gamma:\mathcal{H}\to\mathcal{H}$);}\\
\\
\text{$(2)$ $\gamma$ and $D\gamma$ are bounded and continuous on $\mathcal{H}$ (cf.~\eqref{Dfried} and \eqref{Dfried02}).}
\end{array}\right.
\end{align*}
\end{ass}
\noindent Clearly $(1)$ includes $\sigma$ state dependent.

\begin{remark}
Assumption \ref{ass-sigma} is redundant in the case of constant diffusion coefficients. In fact for any unbounded operator $A$ and any constant $\sigma\in D(A)$ one can pick an orthonormal basis $(\varphi_j)_{j\in\mathbb{N}}$ with $\varphi_1:= \sigma/\|\sigma\|_\mathcal{H}$ and construct $Q$ as in \eqref{bari} with $\lambda_1=1$.
\end{remark}

\begin{example} \emph{In a version of the Musiela model $\mathcal{H}=L^2_\alpha(\mathbb{R}_+)$ is an $L^2$-space with exponential weight $e^{-\alpha x}$. An orthonormal basis $(\varphi_j)_{j\in\mathbb{N}}$ may be constructed from polynomials by using Graham-Schmidt method, and the unbounded operator is $(Af)(x) =f'(x) $ for $f\in D(A) $. The norm $p_j:=\|A\varphi_j\|_\mathcal{H}$ is well defined and finite for all $j\in\mathbb{N}$, hence it is enough to take $\lambda_j:=1/(j \,p_j)^2$ for all $j\ge j_0$ for some $j_0\in\mathbb{N}$, and $\gamma$ according to $(2)$ of Assumption \ref{ass-sigma}.}
\end{example}

\begin{remark}
The second condition in Assumption \ref{ass-sigma} may be substantially relaxed throughout the paper by considering $\gamma$ Lipschitz continuous with sublinear growth, however for simplicity we will not do so.
\end{remark}
Under Assumption \ref{ass-sigma} we have existence and uniqueness of a mild solution $X^x$ to (\ref{SDE-infty}) (cf. \cite{DaPr-Zab}).
From now on and unless otherwise specified (see Section \ref{uniqueness}) we will take Assumptions \ref{ass-psi} and \ref{ass-sigma} as standing assumptions.

Below we obtain some preliminary estimates and some regularity properties of the value function $\mathcal{U}$.
\begin{lemma}\label{lipSDE}
Let $X^x$ and $X^y$ be the mild solutions of (\ref{SDE-infty}) starting at $x$ and $y$, respectively. Then
\begin{eqnarray}
&&\mathbb{E}\left\{\sup_{0\leq t\leq T}\|X^x_t\|^p_{\mathcal{H}}\right\}\leq C_{p,T}(1+\|x\|^p_{\mathcal{H}})\quad1\leq p<\infty,\label{apr1}\\
&&\mathbb{E}\left\{\sup_{0\leq t\leq T}\|X^x_t-X^y_t\|^p_{\mathcal{H}}\right\}\leq C_{p,T}\|x-y\|^p_{\mathcal{H}}\quad1\leq p<\infty,\label{lip}
\end{eqnarray}
where the positive constant $C_{p,T}$ depends only on $p$ and $T$.
\end{lemma}
\begin{proof}
The proof of \eqref{apr1} follows from \cite{DaPr-Zab}, Theorem 7.4, whereas the proof of \eqref{lip} is a consequence of \cite{DaPr-Zab}, Theorem 9.1 and a simple application of Jensen's inequality.
\end{proof}

\begin{prop}\label{reg-v-fun}
The value function $\mathcal{U}(t,x)$ is non-negative, uniformly bounded with the same upper bound of $\Theta$, i.e.
\begin{equation}\label{bdd-V}
\sup_{(t,x)\in[0,T]\times\mathcal{H}}\mathcal{U}(t,x)\leq\overline{\Theta}.
\end{equation}
Moreover, there exists $L_\mathcal{U}>0$ such that
\begin{equation}\label{lip-V}
|\mathcal{U}(t,x)-\mathcal{U}(t,y)|\leq L_\mathcal{U}\|x-y\|_{\mathcal{H}},\qquad t\in[0,T],\,x,y\in\mathcal{H}.
\end{equation}
\end{prop}
\begin{proof}
The first claim is obvious. To show \eqref{lip-V} take $x,y\in\mathcal{H}$ and fix $t\in[0,T]$. Then
\begin{align*}
\mathcal{U}(t,x)-\mathcal{U}(t,y) &\leq \sup_{t\leq\tau\leq T}\mathbb{E}\left\{\Theta(\tau,X^{t,x}_{\tau})-\Theta(\tau,X^{t,y}_{\tau})\right\}\\
&\leq \mathbb{E}\left\{\sup_{t\leq s\leq T}|\Theta(s,X^{t,x}_s)-\Theta(s,X^{t,y}_s)|\right\}\leq L_{\Theta}\,\mathbb{E}\left\{\sup_{t\leq s\leq T}\|X^{t,x}_s-X^{t,y}_s\|_{\mathcal{H}}\right\},
\end{align*}
by \eqref{psi2}. Similarly for $\mathcal{U}(t,y)-\mathcal{U}(t,x)$; hence
\begin{equation*}
|\mathcal{U}(t,x)-\mathcal{U}(t,y)|\leq L_{\Theta}\mathbb{E}\left\{\sup_{t\leq s\leq T}\|X^{t,x}_s-X^{t,y}_s\|_{\mathcal{H}}\right\}.
\end{equation*}
The coefficients in \eqref{SDE-infty} are time-homogeneous, hence
\begin{align*}
\mathbb{E}\left\{\sup_{t\leq s\leq T}\|X^{t,x}_s-X^{t,y}_s\|_{\mathcal{H}}\right\} &= \mathbb{E}\left\{\sup_{0\leq s\leq T-t}\|X^{x}_s-X^{y}_s\|_{\mathcal{H}}\right\}\leq C_{1,T}\|x-y\|_\mathcal{H},
\end{align*}
and \eqref{lip-V} follows with $L_\mathcal{U}=L_{\Theta}\,C_{1,T}$ (cf.~\eqref{lip}).
\end{proof}


\section{The approximation scheme}\label{sec-appr}

In this section we provide an algorithm for the finite dimensional reduction of the optimal stopping problem \eqref{OS1}.
The algorithm requires two separate steps (a similar approach was used for instance in \cite{KeSw03} in a different context). First, we obtain a Yosida approximation of the unbounded operator $A$ by bounded operators $A_\alpha$; then we provide a finite dimensional reduction of the SDE. At each step a corresponding optimal stopping problem is studied.

\subsection{Yosida approximation}\label{sec-yos00}
A natural way to deal with an unbounded linear operator is to introduce its Yosida approximation, which does not require any further assumptions. The Yosida approximation of $A$ is defined as $A_\alpha:=\alpha A(\alpha I-A)^{-1}$, for $\alpha>0$ (cf. \cite{Pazy}). The corresponding SDE is
\begin{equation}\label{SDE-yos}
\left\{
\begin{array}{l}
dX^{(\alpha)x}_t=A_\alpha X^{(\alpha)x}_tdt+\sigma(X^{(\alpha)x}_t)dW^0_t,\qquad t\in[0,T],\\
\\
X^{(\alpha)x}_0=x,
\end{array}
\right.
\end{equation}
which admits a unique strong solution, $X^{(\alpha)x}$, since $A_\alpha$ is a bounded linear operator. That is,
\begin{equation}\label{strongXa}
X^{(\alpha)x}_t=x+\int_0^t{A_\alpha X^{(\alpha)x}_sds}+\int_0^t{\sigma(X^{(\alpha)x}_s)dW^0_s},\qquad t\in[0,T],\,\mathbb{P}\textrm{-a.s.}
\end{equation}
Clearly a strong solution is also a mild solution (cf.~\cite{DaPr-Zab}), hence $X^{(\alpha)x}$ might be equivalently interpreted as
\begin{equation*}
X^{(\alpha)x}_t=e^{tA_{\alpha}}x+\int_0^t{e^{(t-s)A_{\alpha}}\sigma(X^{(\alpha)x}_s)dW^0_s},\qquad t\in[0,T],\,\mathbb{P}\textrm{-a.s.}
\end{equation*}
Similarly $X^{(\alpha)t,x}$ will denote the solution starting at time $t$ from $x$. The following important convergence result is proved in \cite{DaPr-Zab}, Proposition 7.5 and it is here recalled for completeness.
\begin{prop}\label{conv-yos}
Let $X^x$ be the unique mild solution of equation (\ref{SDE-infty}) and $X^{(\alpha)x}$ the unique strong solution of equation (\ref{SDE-yos}). For $1\le p<\infty$, the following convergence holds
\begin{eqnarray*}
\lim_{\alpha\to\infty}{\mathbb{E}\left\{\sup_{0\leq t\leq T}\|X^{(\alpha)x}_t-X^x_t\|^p_{\mathcal{H}}\right\}}=0,\qquad x\in\mathcal{H}.
\end{eqnarray*}
\end{prop}
\noindent
We define $\mathcal{U}_\alpha$ to be the value function of the optimal stopping problem corresponding to $X^{(\alpha)x}$,
\begin{eqnarray}\label{OS2}
\mathcal{U}_\alpha(t,x):=\sup_{t\leq\tau\leq T}\mathbb{E}\left\{\Theta(\tau,X^{(\alpha)t,x}_\tau)\right\}.
\end{eqnarray}
Notice that $\mathcal{U}_\alpha$ satisfies \eqref{bdd-V} and \eqref{lip-V} with the same constants. We have the convergence of $\mathcal{U}_\alpha$ to $\mathcal{U}$ (cf. \eqref{OS1}) as $\alpha\to\infty$ both uniformly with respect to $t$ and in $L^p(0,T;L^p(\mathcal{H},\mu))$-norms.
\begin{theorem}\label{cnv1}
The following convergence results hold,
\begin{equation}\label{cv-1}
\lim_{\alpha\to\infty}\sup_{0\leq t\leq T}|\mathcal{U}_\alpha(t,x)-\mathcal{U}(t,x)|=0,\qquad x\in\mathcal{H},
\end{equation}
\begin{equation}\label{cv-2}
\lim_{\alpha\to\infty}\int_0^T\int_{\mathcal{H}}{|\mathcal{U}_\alpha(t,x)-\mathcal{U}(t,x)|^p\mu(dx)}dt=0,\qquad 1\leq p<\infty.
\end{equation}
\end{theorem}
\begin{proof}
The arguments are similar to those used in the proof of Proposition \ref{reg-v-fun}. In fact by the Lipschitz property of the gain function $\Theta$ and the time-homogeneous character of the processes we have
\begin{eqnarray*}
|\mathcal{U}_\alpha(t,x)-\mathcal{U}(t,x)|&\leq & L_\mathcal{U}\,\mathbb{E}\left\{\sup_{0\leq s\leq T}\left\|X^{(\alpha)x}_s-X^{x}_s\right\|_{\mathcal{H}}\right\}.
\end{eqnarray*}
Since $L_\mathcal{U}$ is independent of $t$, the uniform convergence \eqref{cv-1} follows from Proposition \ref{conv-yos}. To prove \eqref{cv-2} it suffices to apply the dominated convergence theorem, since $\mathcal{U}_\alpha$ is uniformly bounded by $\overline{\Theta}$ (cf.~\eqref{psi1}).
\end{proof}
\begin{coroll}\label{ascoli}
If $\mathcal{U}_\alpha\in C_b([0,T]\times\mathcal{H})$ for all $\alpha>0$, then $\mathcal{U}_\alpha\to \mathcal{U}$ as $\alpha\to\infty$, uniformly on compact subsets  $[0,T]\times\mathcal{K}\subset [0,T]\times\mathcal{H}$. Moreover $\mathcal{U}(t,x)\in C_b([0,T]\times\mathcal{H})$.
\end{coroll}
\begin{proof}
Fix $x\in\mathcal{H}$, then \eqref{cv-1} implies $\mathcal{U}(\,\cdot\,,x)\in C_b([0,T];\mathbb{R})$. For each $\alpha>0$ define
\begin{equation*}
F_\alpha(x):=\sup_{t\in[0,T]}|\mathcal{U}_\alpha(t,x)-\mathcal{U}(t,x)|,
\end{equation*}
then $F_\alpha(x)\to 0$ as $\alpha\to\infty$ by \eqref{cv-1}. The family $(F_\alpha)_{\alpha>0}$ is equibounded and equi-continuous since \eqref{bdd-V} and \eqref{lip-V} hold for both $\mathcal{U}_\alpha$ and $\mathcal{U}$, and
\begin{align*}
\big|F_\alpha(x)-F_\alpha(y)\big| & \leq\sup_{t\in[0,T]}\big|\mathcal{U}_\alpha(t,x)-\mathcal{U}_\alpha(t,y)+\mathcal{U}(t,y)-\mathcal{U}(t,x)\big|\\
& \leq\sup_{t\in[0,T]}\big|\mathcal{U}_\alpha(t,x)-\mathcal{U}_\alpha(t,y)\big|+\sup_{t\in[0,T]}\big|\mathcal{U}(t,y)-
\mathcal{U}(t,x)\big|\leq 2L_\mathcal{U}\|x-y\|_{\mathcal{H}}.
\end{align*}
Then $\mathcal{U}_\alpha$ converges uniformly to $\mathcal{U}$, as $\alpha\to\infty$, on compact subsets $[0,T]\times\mathcal{K}$ (\cite{Die}, Theorem 7.5.6); that is
\begin{eqnarray*}
\lim_{\alpha\to\infty}\sup_{(t,x)\in[0,T]\times\mathcal{K}}|\mathcal{U}_{\alpha}(t,x)-\mathcal{U}(t,x)|=0.
\end{eqnarray*}
Hence, being the uniform limit of bounded continuous functions, $\mathcal{U}$ is continuous on any compact subset $[0,T]\times\mathcal{K}$ (cf. \cite{Die}, Theorem 7.2.1).
That and \eqref{lip-V} imply the continuity of $\mathcal{U}$ on $[0,T]\times\mathcal{H}$.
\end{proof}


\subsection{Finite dimensional reduction}
For each $n\in\mathbb{N}$ let us consider the finite dimensional subset $\mathcal{H}^{(n)}:=span\{\varphi_1,\varphi_2,\ldots,\varphi_n\}$ and the orthogonal projection operator $P_n:\mathcal{H}\to\mathcal{H}^{(n)}$. We approximate the diffusion coefficients of \eqref{SDE-yos}, respectively, by $\sigma^{(n)}:=(P_n\sigma)\circ P_n$ and $A_{\alpha,n}:=P_nA_\alpha P_n$. Notice that $A_{\alpha,n}$ is a bounded linear operator on $\mathcal{H}^{(n)}$. We define the process $X^{(\alpha)x;n}$ as the unique strong solution of the SDE on $\mathcal{H}^{(n)}$ given by
\begin{eqnarray}\label{SDE-n}
\left\{
\begin{array}{l}
dX^{(\alpha)x;n}_t=A_{\alpha,n}X^{(\alpha)x;n}_tdt+\sigma^{(n)}(X^{(\alpha)x;n}_t)dW^0_t+\epsilon_n\sum^n_{i=1}{\varphi_i\,dW^i_t},\qquad t\in[0,T],\\
\\
X^{(\alpha)x;n}_0=P_nx=:x^{(n)},
\end{array}
\right.
\end{eqnarray}
where $(\epsilon_n)_{n}$ is a sequence of positive numbers such that
\begin{eqnarray}\label{qubo}
\sqrt{n}\,\epsilon_n\to0\qquad\textrm{ as}\:\: n\to\infty.
\end{eqnarray}
Obviously $X^{(\alpha)x;n}$ lives in the finite dimensional subspace $\mathcal{H}^{(n)}$ but it may still be seen as a process in $\mathcal{H}$.
\begin{remark}
Notice that at each time $t\in[0,T]$, $X_t^{(\alpha)\,x;n}$ is not the projection of the process $X^{(\alpha)x}_t$ on the finite dimensional subspace. In fact, a process with that property would not be necessarily Markovian. Hence $X^{(\alpha)x;n}$ has to be considered as an auxiliary diffusion process which is used to approximate the original one.
\end{remark}
\begin{prop}\label{c-1}
It holds that
\begin{equation}\label{limit1}
\lim_{n\rightarrow\infty}\mathbb{E}\left\{\sup_{t\in[0,T]}\left\|X^{(\alpha)x;n}_t-X^{(\alpha)x}_t\right\|^2\right\}=0,
\end{equation}
uniformly with respect to $x$ on compact subsets of $\mathcal{H}$.
\end{prop}
\begin{proof}
Since $X^{(\alpha)x;n}$ and $X^{(\alpha)x}$ are both strong solutions, i.e.
\begin{eqnarray*}
X^{(\alpha)x;n}&=&P_nx+\int_0^t{A_{\alpha,n}X^{(\alpha)x;n}_sds}+\int_0^t{\sigma^{(n)}(X^{(\alpha)x;n}_s)dW^0_s}+
\epsilon_n\sum_{i=1}^n{\varphi_iW^i_t},
\end{eqnarray*}
and \eqref{strongXa} holds, we have
\begin{align*}
\|X^{(\alpha)x;n}_t-X^{(\alpha)x}_t\|^2_{\mathcal{H}}\leq & 6\bigg[\|P_nx-x\|^2_{\mathcal{H}}+\Big\|\int_0^t{P_nA_\alpha(X^{(\alpha)x;n}_s-X^{(\alpha)x}_s)ds}\Big\|^2_{\mathcal{H}}\\
&+\Big\|\int_0^t{(I-P_n)A_\alpha X^{(\alpha)x}_sds}\Big\|^2_{\mathcal{H}}+\Big\|\int_0^t{P_n[\sigma(X^{(\alpha)x;n}_s)-\sigma(X^{(\alpha)x}_s)]dW^0_s}\Big\|^2_{\mathcal{H}}\\
&+\Big\|\int_0^t{(I-P_n)\sigma(X^{(\alpha)x}_s)dW^0_s}\Big\|^2_{\mathcal{H}}+\epsilon_n^2\,\sum_{i=1}^n{|W^i_t|^2}\bigg],
\end{align*}
where we used the fact that $A_{\alpha,n}X^{(\alpha)x;n}=P_nA_\alpha X^{(\alpha)x;n}$. Denote by $\|\,\cdot\,\|_{L}$ the norm of linear operators on $\mathcal{H}$. We use H\"older's inequality to estimate the time-integrals, then take the supremum over $t\in[0,T]$ and the expected value. By isometry of the stochastic integral and Fubini's theorem we obtain
\begin{align*}
\mathbb{E}\left\{\sup_{0\leq t\leq T}\|X^{(\alpha)x;n}_t-X^{(\alpha)x}_t\|^2_{\mathcal{H}}\right\}\hspace{-2pt}\leq & 6\bigg[\|P_nx-x\|^2_{\mathcal{H}}+T\|A_\alpha\|^2_{L}\hspace{-3pt}\int_0^T{\hspace{-5pt}\mathbb{E}\left\{\sup_{0\leq u\leq s}\|X^{(\alpha)x;n}_u-X^{(\alpha)x}_u\|^2_{\mathcal{H}}\right\}ds}\\
&+T\int_0^T{\mathbb{E}\left\{\|(I-P_n)A_\alpha X^{(\alpha)x}_s\|^2_{\mathcal{H}}\right\}ds}\\
&+\int_0^T{\mathbb{E}\left\{\|\sigma(X^{(\alpha)x;n}_s)-\sigma(X^{(\alpha)x}_s)\|^2_{\mathcal{H}}\right\}ds}\\
&+\int_0^T{\mathbb{E}\left\{\|(I-P_n)\sigma(X^{(\alpha)x}_s)\|^2_{\mathcal{H}}\right\}ds}+\epsilon_n^2\, \sum_{i=1}^n{\mathbb{E}\{\sup_{0\leq t\leq T}|W^i_t|^2}\}\,\bigg].
\end{align*}

By Assumption \ref{ass-sigma} the diffusion coefficient is Lipschitz and we denote by $L_\sigma>0$ its Lipschitz constant. Then we get
\begin{align*}
\mathbb{E}\left\{\sup_{0\leq t\leq T}\|X^{(\alpha)x;n}_t-X^{(\alpha)x}_t\|^2_{\mathcal{H}}\right\}\hspace{-2pt}\leq & 6\bigg[\|P_nx-x\|^2_{\mathcal{H}}+T\|A_\alpha\|^2_{L}\hspace{-3pt}\int_0^T{\hspace{-5pt}\mathbb{E}\left\{\sup_{0\leq u\leq s}\|X^{(\alpha)x;n}_u-X^{(\alpha)x}_u\|^2_{\mathcal{H}}\right\}ds}\\
&+T\,\int_0^T{\mathbb{E}\left\{\|(I-P_n)A_\alpha X^{(\alpha)x}_s\|^2_{\mathcal{H}}\right\}ds}\\
&+L_\sigma^2\int_0^T{\mathbb{E}\left\{\sup_{0\leq u\leq s}\|X^{(\alpha)x;n}_u-X^{(\alpha)x}_u\|^2_{\mathcal{H}}\right\}ds}\\
&+\int_0^T{\mathbb{E}\left\{\|(I-P_n)\sigma(X^{(\alpha)x}_s)\|^2_{\mathcal{H}}\right\}ds}+\epsilon_n^2\, n\, T\bigg].
\end{align*}
A straightforward application of Gronwall's lemma gives 
\begin{align}\label{roof}
\mathbb{E}&\left\{\sup_{0\leq t\leq T}\|X^{(\alpha)x;n}_t-X^{(\alpha)x}_t\|^2_{\mathcal{H}}\right\}\leq
C_T\,e^{T^2\|A_\alpha\|^2_L+T\,L^2_\sigma}\,M_n(x)
\end{align}
for some positive constant $C_T$ and with
\begin{eqnarray*}
M_n(x):=\|P_nx-x\|^2_{\mathcal{H}}+\epsilon_n^2 n T+\int_0^T{\mathbb{E}\left\{\|(I-P_n)A_\alpha X^{(\alpha)x}_s\|^2_{\mathcal{H}}+\|(I-P_n)\sigma(X^{(\alpha)x}_s)\|^2_{\mathcal{H}}\right\}ds}
\end{eqnarray*}
a continuous real function. The right hand side converges to zero as $n\to\infty$ by dominated convergence and condition \eqref{qubo} on $(\epsilon_n)_n$.
Since $M_n(x)$ decreases to zero as $n\to\infty$, Dini's theorem guarantees uniform convergence on any compact subset $\mathcal{K}\subset\mathcal{H}$.
\end{proof}
\begin{remark}\label{unif}
For any starting time $t\in[0,T]$, the previous proposition and the arguments of its proof still hold for $X^{(\alpha)t,x;n}$ and $X^{(\alpha)t,x}$, thanks to the time-homogeneous property of equations \eqref{SDE-yos} and \eqref{SDE-n}.
\end{remark}
For $n\geq1$ define $\Theta^{(n)}:[0,T]\times\mathcal{H}\to\mathbb{R}$ by
\begin{equation}\label{psi-n}
\Theta^{(n)}(t,x):=\Theta(t,P_nx)=\Theta(t,x^{(n)})
\end{equation}
(cf. \eqref{SDE-n}).
Of course, $P_nx^{(n)}=x^{(n)}$, hence $\Theta^{(n)}(t,\,\cdot\,)=\Theta(t,\,\cdot\,)$ on $\mathcal{H}^{(n)}$. However, in what follows it is convenient to use the notation $\Theta^{(n)}$ since this is a gain function on $\mathcal{H}^{(n)}$ and it will occur in the variational formulation of a finite dimensional optimal stopping problem approximating \eqref{OS2}. It is not hard to see that \eqref{psi2} and Dini's Theorem imply
\begin{equation}\label{psi-n2}
\lim_{n\to\infty}\sup_{(t,x)\in[0,T]\times\mathcal{K}}\big|\Theta^{(n)}(t,x)-\Theta(t,x)\big|=0,\quad\textrm{for every compact}\:
\mathcal{K}\subset\mathcal{H}.
\end{equation}
\begin{remark}\label{isom}
There is an isomorphism $\mathcal{I}_n:(\mathcal{H}^{(n)},\|\,\cdot\,\|_\mathcal{H})\to(\mathbb{R}^n,\|\,\cdot\,\|_{\mathbb{R}^n})$, in fact for any $x\in\mathcal{H}^{(n)}$ we may define $x_i:=\langle x,\varphi_i\rangle_{\mathcal{H}}$, $i=1,2,\ldots n$ and $\mathcal{I}_n x:=(x_1,\ldots,x_n)$.
\end{remark}
Let $\mathcal{U}^{(n)}_\alpha$ be the value function of the optimal stopping problem
\begin{equation}\label{OS3}
\mathcal{U}_\alpha^{(n)}(t,x^{(n)}):=\sup_{t\leq\tau\leq T}\mathbb{E}\left\{\Theta^{(n)}(\tau,X^{(\alpha)t,x;n}_\tau)\right\}.
\end{equation}
Obviously $\mathcal{U}_\alpha^{(n)}$ may also be seen as a function defined on $[0,T]\times\mathbb{R}^{n}$. Again, as for $\mathcal{U}_\alpha$, we point out that $\mathcal{U}^{(n)}_\alpha$ satisfies \eqref{bdd-V} and \eqref{lip-V} with the same constants. The value function $\mathcal{U}^{(n)}_\alpha$ converges to $\mathcal{U}_\alpha$ of \eqref{OS2} as $n\to\infty$. In fact results similar to Theorem \ref{cnv1} and Theorem \ref{ascoli} hold.
\begin{theorem}\label{cnv2}
The following convergence results hold,
\begin{equation}\label{cv-3}
\lim_{n\to\infty}\sup_{(t,x)\in[0,T]\times\mathcal{K}}|\mathcal{U}^{(n)}_\alpha(t,x^{(n)})-\mathcal{U}_\alpha(t,x)|=0,\qquad \mathcal{K}\subset\mathcal{H},\:\:\:\:\mathcal{K}\,\textrm{compact},
\end{equation}
i.e. the convergence is uniform on any compact subset $[0,T]\times\mathcal{K}$, and
\begin{equation}\label{cv-4}
\lim_{n\to\infty}\int_0^T\int_{\mathcal{H}}{|\mathcal{U}^{(n)}_\alpha(t,x^{(n)})-\mathcal{U}_\alpha(t,x)|^p\mu(dx)}dt=0,\qquad 1\leq p<\infty.
\end{equation}
\end{theorem}
\begin{proof}
The proof follows along the same lines as the proof of Theorem \ref{cnv1} since $\Theta^{(n)}(t,X^{(\alpha)t,x;n}_s)=\Theta(t,X^{(\alpha)t,x;n}_s)$, $s\geq t$. Then \eqref{cv-3} follows from the uniform convergence in Proposition \ref{c-1}, and \eqref{cv-4} follows from dominated convergence.
\end{proof}
\noindent
As a consequence we have
\begin{coroll}\label{cont1}
If $\mathcal{U}^{(n)}_\alpha\in C_b([0,T]\times\mathcal{H}^{(n)})$ for all $n\in\mathbb{N}$, then $\mathcal{U}_\alpha\in C_b([0,T]\times\mathcal{H})$.
\end{coroll}
\begin{proof}
Recall that $(\mathcal{U}^{(n)}_\alpha(t,x^{(n)}))_{n}$ is uniformly bounded (cf. Proposition \ref{reg-v-fun}) and \eqref{cv-3} holds. Hence \cite{Die}, Theorem 7.2.1 guarantees the continuity of $\mathcal{U}_\alpha$ on $[0,T]\times\mathcal{K}$. Arguments as in Corollary \ref{ascoli} provide the continuity on $[0,T]\times\mathcal{H}$.
\end{proof}
Later in the paper we will prove that $\mathcal{U}^{(n)}_\alpha$ is indeed continuous (cf.~Theorem \ref{unb-dom}).


\section{Infinite dimensional variational inequality: an existence result}\label{sec-fdvi}

In this Section we prove that the value function $\mathcal{U}$ of \eqref{OS1} is a \emph{strong} solution (in the sense of \cite{Ben-Lio82}) of a parabolic infinite dimensional variational inequality on $[0,T]\times\mathcal{H}$. We start by considering finite-dimensional bounded domains and for those we employ results by \cite{Ben-Lio82}. Then we pass to finite-dimensional unbounded domains, and hence to infinite-dimensional ones by considering solutions in specific Gauss-Sobolev spaces. We deal with uniqueness in Section \ref{uniqueness}.

\subsection{Finite-dimensional, bounded domains: general results}\label{sec:fd01}

When dealing with variational problems on finite dimensional bounded domains, we find bounds which are uniform with respect to the order of the approximation and the size of the domain. Recall the finite dimensional SDE \eqref{SDE-n}. Let $n\in\mathbb{N}$ and fix $\alpha>0$. Let $\mathcal{O}_R$ be the open ball in $\mathbb{R}^n$ with center in the origin and with radius $R$. Define $\tau_{R}(t,x)$ to be the first exit time from $\mathcal{O}_R$, i.e.
\begin{equation}\label{tauR}
\tau_{R}(t,x):=\inf\{s\geq t\,:\,X^{(\alpha)t,x;n}_s\notin \mathcal{O}_R\}\wedge T.
\end{equation}
We are slightly abusing the notation by considering $\mathcal{H}^{(n)}\sim\mathbb{R}^n$ and $X^{(\alpha)t,x;n}\in\mathbb{R}^n$. For simplicity we set $\tau_R:=\tau_R(t,x)$ and we introduce the optimal stopping problem arrested at $\tau_R$,
\begin{equation}\label{OS4}
\mathcal{U}^{(n)}_{\alpha,R}(t,x^{(n)}):=\sup_{t\leq\tau\leq T}\mathbb{E}\left\{\Theta^{(n)}(\tau\wedge\tau_{R},X^{(\alpha)t,x;n}_{\tau\wedge\tau_{R}})\right\}.
\end{equation}
The next result is similar to Theorem \ref{cnv2} and its proof is provided in the Appendix.
\begin{prop}\label{unb-d}
The function $\mathcal{U}^{(n)}_{\alpha,R}$ converges to $\mathcal{U}^{(n)}_{\alpha}$ as $R\to\infty$, uniformly on every compact subset $[0,T]\times\mathcal{K}\subset[0,T]\times\mathbb{R}^n$. Moreover if $\big(\mathcal{U}^{(n)}_{\alpha,R}\big)_{R>0}\subset C_b([0,T]\times\mathbb{R}^n)$, then $\mathcal{U}^{(n)}_{\alpha}\in C_b([0,T]\times\mathbb{R}^n)$.
\end{prop}

Denote by $C^2_c(\mathbb{R}^n)$ the set of all $C^2$-functions on $\mathbb{R}^n$ with compact support. The infinitesimal generator of the diffusion $X^{(\alpha)x;n}$ is
\begin{equation}\label{rullo}
\mathcal{L}_{\alpha,n}\hspace{+1pt}g\hspace{-2pt}:=\frac{1}{2}\epsilon^2_n\sum^n_{i=1}{\frac{\partial^2g}{\partial x_i^2}}+\frac{1}{2}\sum^n_{i,j=1}[\sigma^{(n)}\sigma^{(n)*}]_{i,j}\hspace{-1pt}\frac{\partial^2g}{\partial x_i\partial x_j}+\sum_{i=1}^n{\bigg(\sum_{j=1}^n{x_j\langle A_\alpha\varphi_j ,\varphi_i\rangle}\bigg)\frac{\partial g}{\partial x_i}},
\end{equation}
for $g\in C^2_c(\mathbb{R}^n)$. Notice that
\begin{align}\label{sig1}
[\sigma^{(n)}\sigma^{(n)*}]_{i,j}\hspace{-1pt}(x)=&\langle\sigma^{(n)}(x),\varphi_i\rangle_{\mathcal{H}}\langle\sigma^{(n)}(x),\varphi_j\rangle_{\mathcal{H}},
\end{align}
since $W^0$ is a one dimensional Brownian motion. Moreover $\mathcal{L}_{\alpha,n}$ is a uniformly elliptic operator. The bilinear form associated to the operator $\mathcal{L}_{\alpha,n}$ is
\begin{align*}
a^{(\alpha,n)}_{R}(u,w):=-\hspace{-2pt}\int_{\mathcal{O}_R}{\hspace{-4pt}\mathcal{L}_{\alpha,n}u\,w\,dx^{(n)}}=
\sum_{i,j=1}^n{\left(\int_{\mathcal{O}_R}\frac{1}{2}B^{(n)}_{i,j}\,\frac{\partial u}{\partial x_i}\frac{\partial w}{\partial x_j}dx^{(n)}+\hspace{-4pt}\int_{\mathcal{O}_R}\hspace{-6pt}C^{(n,\alpha)}_{i,j}\frac{\partial u}{\partial x_i} w\,dx^{(n)}\right)},
\end{align*}
for $u,w\in H^1_0(\mathcal{O}_R)$ (cf.~\cite{Ben-Lio82} for the definition of $H^1_0$),
\begin{align}\label{bil-coef00}
\hspace{-6pt}B^{(n)}_{i,j}(x):=\epsilon^2_n\,\delta_{i,j}+[\sigma^{(n)}\sigma^{(n)*}]_{i,j}(x)\quad\text{and}\quad C^{(n,\alpha)}_{i,j}(x):=\frac{1}{2}\frac{\partial [\sigma^{(n)}\sigma^{(n)*}]_{i,j}}{\partial x_j}(x)-x_j\langle A_\alpha\varphi_j ,\varphi_i\rangle
\end{align}
where $\delta_{i,j}=0$ for $i\neq j$ and $\delta_{i,i}=1$. Denote by $(\,\cdot\,,\,\cdot\,)$ the scalar product in $L^2(\mathcal{O}_R)$. From Assumption \ref{ass-sigma} and uniform ellipticity of $\mathcal{L}_{\alpha,n}$, it is not hard to see that there exist constants $\zeta_{\alpha,n,R},C_{\alpha,n,R},C^\prime_{\alpha,n,R}>0$ such that
\begin{eqnarray}
|a^{(\alpha,n)}_{R}(u,w)|&\leq& C_{\alpha,n,R}\|u\|_{H^1_0(\mathcal{O}_R)}\,\|w\|_{H^1_0(\mathcal{O}_R)},\label{bilbd}\\
\nonumber\\
a^{(\alpha,n)}_{R}(u,u)+\zeta_{\alpha,n,R}\:\:(u,u)&\geq& C^\prime_{\alpha,n,R}\|u\|^2_{H^1_0(\mathcal{O}_R)}.\label{coer-1}
\end{eqnarray}
These properties guarantee well-posedness of the variational problem in the following proposition.

Define the closed convex set
\begin{equation}\label{urto}
\mathcal{K}_{n,R}:=\Big\{w\,:\,w\in H^1_0(\mathcal{O}_R)\:\:\textrm{and}\:\:w\geq0\,\textrm{a.e.}\,\Big\},
\end{equation}
and set
\begin{equation}\label{reps1}
u^{(n)}_{\alpha,R}:=\mathcal{U}^{(n)}_{\alpha,R}-\Theta^{(n)},
\end{equation}
\begin{align}\label{gen-inf}
f_{\alpha,n}:&=\frac{\partial \Theta^{(n)}}{\partial t}+\mathcal{L}_{\alpha,n}\Theta^{(n)}\nonumber\\
&=\frac{\partial \Theta^{(n)}}{\partial t}+\frac{1}{2}\epsilon^2_nTr\Big[D^2\Theta^{(n)}\Big]+\frac{1}{2}Tr\big[\sigma^{(n)}\sigma^{(n)*}D^2\Theta^{(n)}\big]+\langle A_{\alpha,n}x,D\Theta^{(n)}\rangle_{\mathcal{H}}.
\end{align}
We expect $u^{(n)}_{\alpha,R}:=\mathcal{U}^{(n)}_{\alpha,R}-\Theta^{(n)}$ to solve an obstacle problem with null obstacle.
Now \eqref{bilbd}, \eqref{coer-1} and the regularity of $f_{\alpha,n}$ in \eqref{gen-inf} are sufficient to apply \cite{Ben-Lio82}, Chapter 3, Theorems 2.2, 2.13, Corollaries 2.2, 2.3, 2.4 to obtain
\begin{prop}\label{bl-str}
There exists a unique solution $\bar{u}$ of the variational problem:
\begin{equation}\label{str-f}
\left\{
\begin{array}{l}
u(t,x^{(n)})\ge0,\:\:(t,x^{(n)})\in[0,T]\times\overline{\mathcal{O}}_{R}\:\:\:\:\text{and}\:\:\:\:
u(T,x^{(n)})=0,\:x^{(n)}\in\overline{\mathcal{O}}_{R};\\
\\
-\Big(\displaystyle{\frac{\partial u}{\partial t}}(t),w-u(t)\Big)+a^{(\alpha,n)}_R(u(t),w-u(t))-(f_{\alpha,n}(t),w-u(t))\geq0\\
\hspace{+190pt}\textrm{for a.e.~$t\in[0,T]$ and for all $w\in\mathcal{K}_{n,R}$}.
\end{array}
\right.
\end{equation}
Moreover, $\bar{u}\in L^p(0,T;W^{1,p}_0(\mathcal{O}_R))\cap L^p(0,T;W^{2,p}(\mathcal{O}_R))$, $\displaystyle{\frac{\partial\bar{u}}{\partial t}}\in L^p(0,T;L^{p}(\mathcal{O}_R))$ for all $1\leq p<\infty$ and $\bar{u}\in C([0,T]\times\overline{\mathcal{O}}_R)$.
\end{prop}
\begin{coroll}\label{str-sol-obs}
The function $\bar{u}$ coincides with the function $u^{(n)}_{\alpha,R}$ and uniquely solves in the almost everywhere sense the obstacle problem
\begin{eqnarray}\label{obs-1}
\left\{
\begin{array}{ll}
\max\left\{\displaystyle{\frac{\partial u}{\partial t}}+\mathcal{L}_{\alpha,n}u+f_{\alpha,n}\,,\,-u\right\}(t,x^{(n)})= 0,\:\: (t,x^{(n)})\in(0,T)\times\mathcal{O}_R, &\\
\\
u(t,x^{(n)})\geq0\:\:\textrm{on}\:\:[0,T]\times\overline{\mathcal{O}}_R\:\:\:\:\text{and}\:\:\:\: u(T,x^{(n)})=0,\:\: x^{(n)}\in\overline{\mathcal{O}}_R;&\\
\\
u(t,x^{(n)})=0,\:\: (t,x^{(n)})\in[0,T]\times{\partial\mathcal{O}_R}. &\\
\end{array}
\right.
\end{eqnarray}
Moreover, the optimal stopping time for $\mathcal{U}^{(n)}_{\alpha,R}$ of \eqref{OS4} is
\begin{equation}
\tau^\star_{\alpha,n,R}:=\inf\{s\geq t\,:\,\mathcal{U}^{(n)}_{\alpha,R}(s,X^{(\alpha)t,x;n}_s)=\Theta^{(n)}(s,X^{(\alpha)t,x;n}_s)\}\wedge
{\tau_R}\wedge T
\end{equation}
and
\begin{eqnarray}\label{dyn-prg-3}
\mathcal{U}^{(n)}_{\alpha,R}(t,x^{(n)})=\mathbb{E}\left\{\mathcal{U}^{(n)}_{\alpha,R}({\tau},X^{(\alpha)t,x;n}
_{{\tau}})\right\}\,,\qquad\text{for all $\tau\leq\tau^\star_{\alpha,n,R}$\,.}
\end{eqnarray}
\end{coroll}
\noindent The proof follows from Proposition \ref{bl-str} and is outlined in the Appendix for completeness.

\begin{remark}\label{weaker}
Notice that when $\Theta$ fulfils only \eqref{condth01}, the variational inequality still makes sense by considering $f_{\alpha,n}$ as a map from $[0,T]$ to the dual space of $W^{1,p}$.
\end{remark}


\subsubsection{Penalization method and some uniform bounds}

Now we would like to take limits in the variational inequalities as $R\to\infty$, $n\to\infty$, $\alpha\to\infty$, respectively. For that we need bounds on $u^{(n)}_{\alpha,R}$, $Du^{(n)}_{\alpha,R}$ and $\frac{\partial}{\partial\,t}u^{(n)}_{\alpha,R}$ uniformly in $(R,n,\alpha)$. The first two bounds are obtained in the next
Proposition.

Recall Remark \ref{conn} and the definition of $W^{1,p}(\mathcal{H},\mu)$ of \eqref{def-W1p}. Then for each $R>0$, consider the zero extension outside $\mathcal{O}_R$ of $u^{(n)}_{\alpha,R}$ and still denote it by $u^{(n)}_{\alpha,R}$ for simplicity.
\begin{prop}\label{univ-est-sol}
The family $\big(u^{(n)}_{\alpha,R}\big)_{R,n,\alpha}$ is bounded in $L^p(0,T;W^{1,p}(\mathcal{H},\mu))$ for $1\le p<+\infty$ uniformly with respect to $(R,n,\alpha)\in(0,+\infty)\times\mathbb{N}\times(0,+\infty)$.
\end{prop}
\begin{proof}
Clearly we may think of $u^{(n)}_{\alpha,R}$ as a function defined on $[0,T]\times\mathcal{H}$. Then from Assumption \ref{ass-psi} and \eqref{reps1} it follows that ${u}^{(n)}_{\alpha,R}$ is bounded by $2\overline{\Theta}$ for all $(R,n,\alpha)\in(0,+\infty)\times\mathbb{N}\times(0,+\infty)$, uniformly in $(t,x)\in[0,T]\times\mathcal{H}$; i.e.~$\|{u}^{(n)}_{\alpha,R}(t)\|_{L^p(\mathcal{H},\mu)}\le2\overline{\Theta}$, $t\in[0,T]$. It is easy to see that
\begin{align}\label{bound00}
\big\|u^{(n)}_{\alpha,R}\big\|_{L^p(0,T;L^p(\mathcal{H},\mu))}=
\left(\int_0^T{\big\|u^{(n)}_{\alpha,R}(t)\big\|^p_{L^p(\mathcal{H},\mu)}dt}\right)^{\frac{1}{p}}\le 2\:\overline{\Theta}\:T^{\frac{1}{p}},\quad1\le p<+\infty.
\end{align}
Moreover, for all $(R,n,\alpha)\in(0,+\infty)\times\mathbb{N}\times(0,+\infty)$, $u^{(n)}_{\alpha,R}$ is Lipschitz in the space variable, uniformly with respect to $t\in[0,T]$, with Lipschitz constant lesser or equal than $L_\mathcal{U}+L_\Theta$. It follows that $\big\|Du^{(n)}_{\alpha,R}(t,x^{(n)})\big\|_{\mathcal{H}}=\big\|Du^{(n)}_{\alpha,R}(t,x^{(n)})\big\|_{\mathbb{R}^n}\le L_\mathcal{U}+L_\Theta$ for a.e.~$(t,x^{(n)})\in[0,T]\times\mathbb{R}^{n}$.
Since $\mu$ restricted to $\mathbb{R}^n$ is equivalent to the Lebesgue measure (cf.~Remark \ref{conn}) it follows that $\big\|Du^{(n)}_{\alpha,R}(t)\big\|_{L^p(\mathcal{H},\mu;\mathcal{H})}\le L_\mathcal{U}+L_\Theta$, $t\in[0,T]$ and
\begin{align}\label{bound01}
\big\|Du^{(n)}_{\alpha,R}\big\|_{L^p(0,T;L^p(\mathcal{H},\mu;\mathcal{H}))}=
\left(\int_0^T{\big\|Du^{(n)}_{\alpha,R}(t)\big\|^p_{L^p(\mathcal{H},\mu;\mathcal{H})}dt}\right)^{\frac{1}{p}}\le (L_\mathcal{U}+L_\Theta)T^{\frac{1}{p}},\quad1\le p<+\infty.
\end{align}
\end{proof}
We now go through a number of steps (including penalization) in order to find a bound on $\frac{\partial}{\partial\,t}\,u^{(n)}_{\alpha,R}$. First, by arguments as in \cite{Men80} we have
\begin{lemma}\label{bound-f}
Let $\nu$ be any real adapted process in $[0,1]$, $\varepsilon>0$, $t\in[0,T]$, $x^{(n)}$ and $y^{(n)}$ in $\mathbb{R}^n$, then
\begin{align}\label{bound-f00}
\left|\mathbb{E}\left\{\int^{\tau^x_R}_{\tau^x_R\wedge\tau^y_R}{e^{-\int^s_t{\frac{1}
{\varepsilon}\nu(u)du}}f_{\alpha,n}(s,X^{(\alpha)t,x;n}_s)ds}\right\}\right|\le L_f\big\|x^{(n)}-y^{(n)}\big\|_{\mathcal{H}}
\end{align}
where $\tau^x_R:=\tau_R(t,x)$ and $\tau^y_R:=\tau_R(t,y)$ (cf.~\eqref{tauR}) and $L_f>0$ only depends on $L_\Theta$ and $L_\mathcal{U}$ (cf.~Assumption \ref{ass-psi} and Proposition \ref{reg-v-fun}).
\end{lemma}
\begin{proof}
The proof is in the Appendix.
\end{proof}
Now we need to recall the penalization method used in \cite{Ben-Lio82}, Chapter 3, Section 2, to obtain existence and uniqueness results for parabolic variational inequalities as in our Proposition \ref{bl-str}. For fixed $(R,n,\alpha)$ we denote $u^{R}:=u^{(n)}_{\alpha,R}$ to simplify notation. In \cite{Ben-Lio82} $u^{R}$ is found in the limit as $\varepsilon\to0$ of functions $u^{R}_{\varepsilon}$ solving the penalized problem
\begin{align}\label{penal00}
\left\{
\begin{array}{l}
u^{R}_{\varepsilon}\in L^2(0,T;H^2(\mathcal{O}_R))\cap L^2(0,T;H^1_0(\mathcal{O}_R))\,;\:\:\:\frac{\partial}{\partial\,t}u^{R}_{\varepsilon}\in L^2(0,T;L^2(\mathcal{O}_R)); \\
\\
\frac{\partial\,u^{R}_{\varepsilon}}{\partial\,t}+\mathcal{L}_{\alpha,n}u^{R}_{\varepsilon}=-f_{\alpha,n}-
\frac{1}{\varepsilon}\big[-u^{R}_{\varepsilon}\big]^+,\quad\text{for a.e.~$(t,x^{(n)})\in[0,T]\times\mathcal{O}_R$}\\
\\
u^{R}_{\varepsilon}(T,x^{(n)})=0,\:\:\:\text{for $x^{(n)}\in\mathcal{O}_R$.}
\end{array}
\right.
\end{align}
From now on we consider the zero extension outside $\mathcal{O}_R$ of $u^R_\varepsilon$ which we still denote by $u^R_\varepsilon$. Then $u^R_\varepsilon$ may be represented as (cf.~\cite{Ben-Lio82}, Chapter 3, Section 4, Theorem 4.4)
\begin{align}\label{penal01}
u^{R}_{\varepsilon}(t,x^{(n)})=\sup_{\nu\in[0,1]}\mathbb{E}\left\{\int^{\tau^x_R}_t{e^{-\int^s_t{\frac{1}
{\varepsilon}\nu(u)du}}f_{\alpha,n}(s,X^{(\alpha)t,x;n}_s)\,ds}\right\},
\end{align}
where the supremum is taken over all real adapted stochastic processes $\nu\in[0,1]$. Lipschitz continuity of $u^R_\varepsilon$ in the space variable, uniformly with respect to time, follows by means of Lemma \ref{bound-f}. The proof is inspired by \cite{Men80} and it is contained in the Appendix \ref{bound-f}.
\begin{lemma}\label{lip-ue}
There exists a constant $L_P>0$ independent of $(\varepsilon,R,\alpha,n)$ such that
\begin{align}\label{penal02}
\big\|Du^{R}_{\varepsilon}(t,x^{(n)})\big\|_{\mathcal{H}}=\big\|Du^{R}_{\varepsilon}(t,x^{(n)})\big\|_{\mathbb{R}^n}
\le L_P\qquad\text{for a.e.~$(t,x^{(n)})\in[0,T]\times\mathbb{R}^n$.}
\end{align}
\end{lemma}

In order to get bounds in $L^p(\mathcal{H},\mu)$ it is convenient to find a formulation of \eqref{penal00} in such space. To do so we introduce some notation (cf.~Remark \ref{conn}).
\begin{definition}\label{b-sp}
For $1< p<\infty$ and $p'$ such that $\frac{1}{p}+\frac{1}{p'}=1$, denote by $\mathcal{V}^p_n$ the space
\begin{equation}\label{b-sp1}
\mathcal{V}^p_n:=\{v\,:\,v\in L^{2p}(\mathbb{R}^n,\mu_n)\:\textit{and}\:\:D v\in L^{2p'}(\mathbb{R}^n,\mu_n;\mathbb{R}^n)\}
\end{equation}
endowed with the norm
\begin{eqnarray}
\lnr v\rnr_{p,n}:=\|v\|_{L^{2p}(\mathbb{R}^n,\mu_n)}+\|D v\|_{L^{2p'}(\mathbb{R}^n,\mu_n;\mathbb{R}^n)}.
\end{eqnarray}
\end{definition}
\noindent Then $(\mathcal{V}^p_n,\lnr\cdot\rnr_{p,n})$ is a separable Banach space.

Denote by $(\,\cdot\,,\,\cdot\,)_{\mu_n}$ the scalar product in $L^2(\mathbb{R}^n,\mu_n)$ and, for $u,w\in\mathcal{V}^p_n$, define the bilinear form associated to the operator $\mathcal{L}_{\alpha,n}$ (cf.\ \eqref{rullo}),
\begin{align}\label{lup1}
a^{(\alpha,n)}_{\mu}&(u,w):=-\int_{\mathbb{R}^n}{\mathcal{L}_{\alpha,n}u\,w\mu_n(dx)}=
\sum_{i,j=1}^n{\left(\int_{\mathbb{R}^n}\frac{1}{2}B^{(n)}_{i,j}\,\frac{\partial u}{\partial x_i}\frac{\partial w}{\partial x_j}\mu_n(dx)+\hspace{-4pt}\int_{\mathbb{R}^n}\hspace{-6pt}\overline{C}^{(n,\alpha)}_{i,j}\frac{\partial u}{\partial x_i} w\,\mu_n(dx)\right)}
\end{align}
with
\begin{align}\label{bil-coef01}
\hspace{-6pt}\overline{C}^{(n,\alpha)}_{i,j}(x):=C^{(n,\alpha)}_{i,j}(x)-\frac{1}{2}\left(
[\sigma^{(n)}\sigma^{(n)*}]_{i,j}\,\frac{x_j}{\lambda_j}+\epsilon_n^2\,\delta_{i,j}\,\frac{x_j}{\lambda_j}\right)
\end{align}
and $B^{(n)}_{i,j}$ and $C^{(n,\alpha)}_{i,j}$ as in \eqref{bil-coef00}. From \eqref{sig1} it follows that
\begin{align}\label{bra}
\hspace{-5pt}\frac{\partial}{\partial x_j}[\sigma^{(n)}\sigma^{(n)*}(x)]_{i,j}\hspace{-1pt} =&\langle D\sigma^{(n)}(x)\varphi_j,\varphi_i\rangle_{\mathcal{H}} \langle\sigma^{(n)}(x),\varphi_j\rangle_{\mathcal{H}} + \hspace{-1pt}\langle  D\sigma^{(n)}(x)\varphi_j,\varphi_j\rangle_{\mathcal{H}} \langle\sigma^{(n)}(x),\varphi_i\rangle_{\mathcal{H}},
\end{align}
then \eqref{bra} and the isometry $\mathcal{H}^{(n)}\sim\mathbb{R}^n$ allow us to rewrite the bilinear form \eqref{lup1} as
\begin{align}\label{lup2}
a^{(\alpha,n)}_{\mu}&(u,w):=\int_{\mathbb{R}^n}\frac{1}{2}\langle B^{(n)}Du,D w\rangle _\mathcal{H}\,\mu_n(dx)+\hspace{-4pt}\int_{\mathbb{R}^n}\hspace{-6pt}\langle\, \overline{C}^{(n,\alpha)},D u\rangle_{\mathcal{H}}\, w\,\mu_n(dx)
\end{align}
where $B^{(n)}:=\sigma^{(n)}\sigma^{(n)*}+\epsilon_n^2\,I\in\mathcal{L}(\mathcal{H})$, the set of all linear operators on $\mathcal{H}$, and $\overline{C}^{(n,\alpha)}\in\mathcal{H}$ is given by
\begin{align*}
\overline{C}^{(n,\alpha)}:=\frac{1}{2}\left(Tr[D\sigma^{(n)}]_{\mathcal{H}}\sigma^{(n)}+D\sigma^{(n)}\cdot\sigma^{(n)}-
2A_{\alpha,n}x-\sigma^{(n)}\sigma^{(n)*}Q_n^{-1}x-\epsilon_n^2\,Q_n^{-1}x\right).
\end{align*}
Here $Q_n:=P_nQP_n$ and $(D\sigma^{(n)}\cdot\sigma^{(n)})_i:=\sum_{j=1}^n{(D\sigma^{(n)})_{i,j}\,\sigma^{(n)}_j}$, $i=1,\ldots,n$. The continuity in $\mathcal{V}^p_n$ of the bilinear form \eqref{lup2} follows from the next result which makes use of Assumption \ref{ass-Q}.
\begin{theorem}\label{uni-est}
For every $1<p<\infty$ there exists a constant $C_{\mu,\gamma,p}>0$, depending only on $\mu$, $p$ and the bounds of $\gamma$ in Assumption \ref{ass-sigma}, such that
\begin{equation}\label{univ}
\int_0^T{\hspace{-5pt}|a^{(\alpha,n)}_\mu(u(t),w(t))| dt}\leq C_{\mu,\gamma,p}\left(\int_0^T{\lnr u(t)\rnr^2_{p,n}}dt\right)^\frac{1}{2}\left(\int_0^T{\lnr w(t)\rnr^2_{p,n}}dt\right)^\frac{1}{2}
\end{equation}
for all $u,w\in L^2(0,T;\mathcal{V}^p_n)$.
\end{theorem}
\begin{proof}
Thanks to Assumption \ref{ass-sigma} and since $Q$ is of trace class (cf.~\eqref{bari}) the estimate is straightforward for all the terms in \eqref{lup2} except those involving $\epsilon^2_nQ^{-1}_n$ and $A_{\alpha,n}$. As for the first case notice that, although $Q^{-1}_n$ becomes unbounded as $n\to\infty$, there is no restriction in assuming that the sequence $(\epsilon_n)_{n\in\mathbb{N}}$ is such that $\epsilon_n Q^{-1}_n\to0$ as $n\to\infty$ (cf.~\eqref{qubo}). It then remains to look at
\begin{eqnarray}
(I):\hspace{-8pt}&=&\hspace{-8pt}\bigg|\int_{\mathbb{R}^n}{\langle A_{\alpha,n}x,Du\rangle_{\mathcal{H}}\, w\,  \mu_n(dx)}\bigg|.\label{(I)}
\end{eqnarray}
Recalling Assumption \ref{ass-Q} and using H\"older's inequality we obtain
\begin{align}\label{stimaA}
(I)\leq&\left(\int_{\mathbb{R}^n}\|A_{\alpha,n} x\|^2_{\mathcal{H}}\mu_n(dx)\right)^{\frac{1}{2}}\left(\int_{\mathbb{R}^n}\|D u\|^2_{\mathcal{H}}| w|^2 \mu_n(dx)\right)^{\frac{1}{2}}\\
\le&\left(\sum^{n}_{j=1}\int_{\mathbb{R}^n}\big|\langle x,A^*_{\alpha,n}\varphi_j\rangle\big|^2_{\mathcal{H}}\mu_n(dx)\right)^{\frac{1}{2}}\lnr u\rnr_{p,n}\lnr w\rnr_{p,n}\le\left(Tr[AQA^*]\right)^{\frac{1}{2}}\lnr u\rnr_{p,n}\lnr w\rnr_{p,n}\,\,,\nonumber
\end{align}
where the last inequality follows from $\int_{\mathbb{R}^n}\big|\langle x,y\rangle\big|^2_{\mathcal{H}}\mu_n(dx)=\langle Q_ny,y\rangle_{\mathcal{H}}$ for $y\in\mathcal{H}$ (see for instance \cite{DaPr}, p.13).
\end{proof}

For $v^R\in H^1_0(\mathcal{O}_R)$ we consider its zero extension outside $\mathcal{O}_R$, again denoted by $v^R$. Multiplying \eqref{penal00} by $v^R\,\frac{1}{\sqrt{(2\pi)^n\lambda_1\lambda_2\cdots\lambda_n}}\exp{
\left(-\sum_{i=1}^n{\frac{x_i^2}{\lambda_i}}\right)}$ and integrating by parts over $\mathbb{R}^n$ gives the penalized problem in a weaker form; that is
\begin{align}\label{penal03}
-\big(\frac{\partial\,u^R_\varepsilon}{\partial\,t}(t),v^R\big)_{\mu_n}+a^{(\alpha,n)}_\mu(u^R_\varepsilon(t),v^R)-
\frac{1}{\varepsilon}\big(\big[-u^R_\varepsilon(t)\big]^+,v^R\big)_{\mu_n}=\big(f_{\alpha,n}(t),v^R\big)_{\mu_n}\quad t\in[0,T].
\end{align}
Following arguments as in \cite{Ben-Lio82}, Chapter 3, Section 2, p.~246, we finally obtain a bound on $\frac{\partial}{\partial\,t}u^{(n)}_{\alpha,R}$.
\begin{prop}\label{bound-ddt}
The family $\big(\frac{\partial}{\partial\,t}\,u^{(n)}_{\alpha,R}\big)_{R,n,\alpha}$ is bounded in $L^2(0,T;L^2(\mathcal{H},\mu))$, uniformly with respect to $(R,n,\alpha)\in(0,\infty)\times\mathbb{N}\times(0,\infty)$.
\end{prop}
\begin{proof}
As in \cite{Ben-Lio82} one may take $v^R=\frac{\partial}{\partial\,t}u^R_\varepsilon$, possibly up to a regularization, or considering finite differences, as the estimate obtained at the end does not involve second derivatives of $u^R_\varepsilon$ and it is therefore consistent. Plugging such $v^R$ in \eqref{penal03} gives
\begin{align}\label{penal04}
-\Big\|\frac{\partial\,u^R_\varepsilon}{\partial\,t}\Big\|^2_{L^2(\mathcal{H},\mu)}(t)
+a^{(\alpha,n)}_\mu(u^R_\varepsilon,\frac{\partial\,u^R_\varepsilon}{\partial\,t})(t)+
\frac{1}{2\varepsilon}\frac{\partial}{\partial\,t}\big\|\big[-u^R_\varepsilon\big]^+\big\|^2_{L^2(\mathcal{H},\mu)}(t)
=\big(f_{\alpha,n},\frac{\partial\,u^R_\varepsilon}{\partial\,t}\big)_{\mu_n}(t).
\end{align}
Next observe that \eqref{lup2} implies
\begin{align}\label{penal05}
a^{(\alpha,n)}_\mu(u^R_\varepsilon,\frac{\partial\,u^R_\varepsilon}{\partial\,t})=
a^{(\alpha,n)}_{\mu,0}(u^R_\varepsilon,\frac{\partial\,u^R_\varepsilon}{\partial\,t})+
a^{(\alpha,n)}_{\mu,1}(u^R_\varepsilon,\frac{\partial\,u^R_\varepsilon}{\partial\,t})
\end{align}
where
\begin{align}\label{penal06}
a^{(\alpha,n)}_{\mu,0}(u^R_\varepsilon,\frac{\partial\,u^R_\varepsilon}{\partial\,t})=&\int_{\mathbb{R}^n}\frac{1}{2}
\langle B^{(n)}Du^R_\varepsilon,\frac{\partial}{\partial\,t}D u^R_\varepsilon\rangle _\mathcal{H}\,\mu_n(dx)\\
=&\frac{\partial}{\partial\,t}\int_{\mathbb{R}^n}
\langle B^{(n)}Du^R_\varepsilon,D u^R_\varepsilon\rangle _\mathcal{H}\,\mu_n(dx)=2\,\frac{\partial}{\partial\,t}a^{(\alpha,n)}_{\mu,0}
(u^R_\varepsilon,u^R_\varepsilon),\nonumber
\end{align}
by symmetry and
\begin{align}\label{penal07}
a^{(\alpha,n)}_{\mu,1}(u^R_\varepsilon,\frac{\partial\,u^R_\varepsilon}{\partial\,t})=\int_{\mathbb{R}^n}\hspace{-6pt}\langle\, \overline{C}^{(n,\alpha)},D u^R_\varepsilon\rangle_{\mathcal{H}}\, \frac{\partial\,u^R_\varepsilon}{\partial\,t}\,\mu_n(dx).
\end{align}
By integrating with respect to $t$ over $[0,T]$, recalling that $u^R_\varepsilon(T,\,\cdot\,)=0$ and rearranging terms one obtains
\begin{align}\label{penal08}
&\Big\|\frac{\partial\,u^R_\varepsilon}{\partial\,t}\Big\|^2_{L^2(0,T;L^2(\mathcal{H},\mu))}+
2a^{(\alpha,n)}_{\mu,0}
\big(u^R_\varepsilon(0),u^R_\varepsilon(0)\big)+
\frac{1}{2\varepsilon}\big\|\big(-u^R_\varepsilon\big)^+\big\|^2_{L^2(\mathcal{H},\mu)}(0)\\
&=-\int^T_0\big(f_{\alpha,n},\frac{\partial\,u^R_\varepsilon}{\partial\,t}\big)_{\mu_n}(t)d\,t+2a^{(\alpha,n)}_{\mu,0}
\big(u^R_\varepsilon(T),u^R_\varepsilon(T)\big)+\int^T_0a^{(\alpha,n)}_{\mu,1}
(u^R_\varepsilon,\frac{\partial\,u^R_\varepsilon}{\partial\,t})(t)d\,t\nonumber
\end{align}
and therefore
\begin{align}\label{penal09}
&\Big\|\frac{\partial\,u^R_\varepsilon}{\partial\,t}\Big\|^2_{L^2(0,T;L^2(\mathcal{H},\mu))}\\
&\le
\left|\int^T_0\big(f_{\alpha,n},\frac{\partial\,u^R_\varepsilon}{\partial\,t}\big)_{\mu_n}(t)d\,t\right|+2\left|
a^{(\alpha,n)}_{\mu,0}
\big(u^R_\varepsilon(T),u^R_\varepsilon(T)\big)\right|+\left|\int^T_0a^{(\alpha,n)}_{\mu,1}
(u^R_\varepsilon,\frac{\partial\,u^R_\varepsilon}{\partial\,t})(t)d\,t\right|.\nonumber
\end{align}
To provide estimates for the terms on the right-hand side of \eqref{penal10}, notice that by Assumption \ref{ass-sigma} and Lemma \ref{lip-ue}, one gets
\begin{align}\label{penal10}
\left|
a^{(\alpha,n)}_{\mu,0}
\big(u^R_\varepsilon(T),u^R_\varepsilon(T)\big)\right|\le C_1
\end{align}
with $C_1>0$ depending only on $L_P$, $\mu$ and the bounds on $\gamma$. Also, Assumption \ref{ass-Q} and arguments as in the proof of Theorem \ref{uni-est} give
\begin{align}\label{penal11}
&\left|\int^T_0a^{(\alpha,n)}_{\mu,1}
(u^R_\varepsilon,\frac{\partial\,u^R_\varepsilon}{\partial\,t})(t)d\,t\right|\nonumber\\
&\le\int^T_0\int_{\mathbb{R}^n}
\hspace{-3pt}\big\|\overline{C}^{(n,\alpha)}\big\|_{\mathcal{H}}\,\big\|D u^R_\varepsilon\big\|_{\mathcal{H}}\,\Big|\frac{\partial\,u^R_\varepsilon}{\partial\,t}\Big|\mu_n(dx)d\,t\le L_P\,C_2
\Big\|\frac{\partial\,u^R_\varepsilon}{\partial\,t}\Big\|_{L^2(0,T;L^2(\mathcal{H},\mu))}
\end{align}
with $C_2>0$ depending only on $\mu$, $T$ and the bounds on $\gamma$. Similarly Assumption \ref{ass-psi} implies
\begin{align}\label{penal12}
\left|\int^T_0\big(f_{\alpha,n},\frac{\partial\,u^R_\varepsilon}{\partial\,t}\big)_{\mu_n}(t)d\,t\right|\le C_3
\Big\|\frac{\partial\,u^R_\varepsilon}{\partial\,t}\Big\|_{L^2(0,T;L^2(\mathcal{H},\mu))}
\end{align}
with $C_3>0$ depending only on $\mu$, $T$, $L_\Theta$, $L'_\Theta$ and the bounds on $\gamma$.

Therefore, from \eqref{penal09}, \eqref{penal10}, \eqref{penal11} and \eqref{penal12} it follows that
\begin{align}\label{penal13}
\Big\|\frac{\partial\,u^R_\varepsilon}{\partial\,t}\Big\|_{L^2(0,T;L^2(\mathcal{H},\mu))}\le C_4
\end{align}
for a suitable $C_4>0$ independent of $(\varepsilon,R,n,\alpha)$. Now, \eqref{penal13} holds for $\frac{\partial}{\partial\,t}u^R$ as well since it is obtained as the weak limit in $L^2(0,T;L^2(\mathcal{H},\mu))$ of $\frac{\partial}{\partial\,t}u^R_\varepsilon$ as $\varepsilon\to\infty$ (cf.~\cite{Ben-Lio82}, Chapter 3, Section 2.3, p.~239).
\end{proof}

\subsection{Finite-dimensional unbounded domains}\label{sec:fd02}

Recall the optimal stopping problem \eqref{OS3} and set
\begin{equation}\label{raid}
u^{(n)}_\alpha:=\mathcal{U}^{(n)}_\alpha-\Theta^{(n)}.
\end{equation}
From Proposition \ref{unb-d}, Proposition \ref{univ-est-sol} and Proposition \ref{bound-ddt} it follows
\begin{lemma}\label{weak01}
There exists a sequence $(R_i)_{i\in\mathbb{N}}$ such that $R_i\to\infty$ as $i\to\infty$ and $u^{(n)}_{\alpha,R_i}$ converges to $u^{(n)}_\alpha$ as $R_i\to\infty$, weakly in $L^p(0,T;\mathcal{V}^p_n)$ and strongly in $L^p(0,T; L^p(\mathbb{R}^n,\mu_n))$, $1\leq p<\infty$. Moreover, $\frac{\partial\,}{\partial\,t}u^{(n)}_{\alpha,R_i}$ converges to $\frac{\partial\,}{\partial\,t}u^{(n)}_{\alpha}$ as $R_i\to\infty$, weakly in $L^2(0,T;L^2(\mathbb{R}^n,\mu_n))$.
\end{lemma}

In the spirit of \cite{Ben-Lio82}, Chapter 3, Section 1.11, take $w_R\in\mathcal{K}_{n,R}$ (cf.\ \eqref{urto}) and recall that $u^{(n)}_{\alpha,R}$ is the unique solution of \eqref{str-f}. Define $\tilde{w}_R\in\mathcal{K}_{n,R}$ by
\begin{equation}\label{subs}
\tilde{w}_R-u^{(n)}_{\alpha,R}:=\frac{1}{\sqrt{(2\pi)^n\lambda_1\lambda_2\cdots\lambda_n}}\exp{
\left(-\sum_{i=1}^n{\frac{x_i^2}{\lambda_i}}\right)}(w_R-u^{(n)}_{\alpha,R}).
\end{equation}
Take $w=\tilde{w}_R$ in \eqref{str-f} and use \eqref{subs} to obtain
\begin{align}\label{seq-vi}
-\big(\frac{\partial\,u^{(n)}_{\alpha,R}}{\partial t},w_R-u^{(n)}_{\alpha,R}\big)_{\mu_n}\hspace{-2pt}+\hspace{-2pt}a^{(\alpha,n)}_\mu\big(u^{(n)}_{\alpha,R}\,,\,
w_R-u^{(n)}_
{\alpha,R}\big)\hspace{-1pt}-\hspace{-1pt}(f_{\alpha,n}\,,\,&w_R-u^{(n)}_{\alpha,R})_{\mu_n}\geq0\quad\text{for a.e.~$t\in[0,T]$}.
\end{align}

For every $1<p<\infty$, denote by $\mathcal{K}^p_{n,\mu}$ the closed convex set
\begin{equation}\label{ccs}
\mathcal{K}^p_{n,\mu}:=\{w:\,w\in \mathcal{V}^p_n\:\:\text{and}\:\:w\geq0\,\textrm{a.e. in}\,\}.
\end{equation}
We can now extend Proposition \ref{bl-str} to the unbounded case, i.e.\ to $\mathbb{R}^n$.
\begin{theorem}\label{unb-dom}
For every $1<p<\infty$, the function $u^{(n)}_\alpha$ is a solution of the variational problem on $\mathbb{R}^n$
\begin{eqnarray}\label{d-2}
\left\{
\begin{array}{l}
u\in L^2(0,T;\mathcal{V}^p_n);\:\:\:\:\frac{\partial}{\partial\,t}u\in L^2(0,T;L^2(\mathbb{R}^n,\mu_n));\\
\\
u(T,x^{(n)})=0,\:\text{for}\:\:x^{(n)}\in\mathbb{R}^n;\:\:\:\:u(t,x^{(n)})\geq 0,\:\text{for}\:\:(t,x^{(n)})\in[0,T]\times\mathbb{R}^n;\\
\\
\hspace{-3pt}\displaystyle{-\big(\frac{\partial u}{\partial t}(t),w-u(t)\big)_{\mu_n}+a^{(\alpha,n)}_\mu\big(u(t),w-u(t)\big)- \big(f_{\alpha,n}(t),w-u(t)\big)_{\mu_n}\geq0},\\
\hspace{+230pt}\text{for a.e.~$t\in[0,T]$ and for all $w\in \mathcal{K}^p_{n,\mu}$}.
\end{array}\right.
\end{eqnarray}

Moreover, $u^{(n)}_\alpha\in C([0,T]\times\mathbb{R}^n)$ and the optimal stopping time for $\mathcal{U}^{(n)}_\alpha$ of \eqref{OS3} is
\begin{equation}\label{bis}
\tau^\star_{\alpha,n}(t,x):=\inf\{s\geq t\,:\,\mathcal{U}^{(n)}_\alpha(s,X^{(\alpha)t,x;n}_s)=\Theta^{(n)}(s,X^{(\alpha)t,x;n}_s)\}\wedge T.
\end{equation}
\end{theorem}
\begin{proof}
Observe that, by arguments on cut-off functions as in \cite{Adams}, Theorem 3.22, for each $w\in\mathcal{K}^p_{n,\mu}$ there exists a family $(w_{R})_{R>0}\subset\mathcal{K}^p_{n,\mu}\cap\mathcal{K}_{n,R}$ (cf.~\eqref{urto}) such that $w_{R}\to w$ as $R\to\infty$ in $\mathcal{V}^p_n$. Rewrite the inequality \eqref{seq-vi} as
\begin{align}\label{d-1}
-\big(\frac{\partial u^{(n)}_{\alpha,R}}{\partial t}(t),w_{R}-u^{(n)}_{\alpha,R}(t)\big)_{\mu_n}&+a^{(\alpha,n)}_\mu\big(u^{(n)}_{\alpha,R}(t),w_{R}\big) \\
&\geq \big(f_{\alpha,n}(t),w_{R}-u^{(n)}_{\alpha,R}(t)\big)_{\mu_n}+a^{(\alpha,n)}_\mu\big(u^{(n)}_{\alpha,
R}(t),u^{(n)}_{\alpha,R}(t)\big).\nonumber
\end{align}
Consider the sequences $(R_i)_{i\in\mathbb{N}}$ and $(u^{(n)}_{\alpha,R_i})_{i\in\mathbb{N}}$ of Lemma \ref{weak01} and fix arbitrary $0\le t_1<t_2\le T$. Then taking limits as $i\to\infty$ gives (cf.~for instance \cite{Brezis}, Proposition 3.5)
\begin{eqnarray}
\int_{t_1}^{t_2}{(\frac{\partial u^{(n)}_{\alpha, R_i}}{\partial t}, w_{R_i}-u_{\alpha,R_i}^{(n)})_{\mu_n}dt} &\to& \int_{t_1}^{t_2}{(\frac{\partial u^{(n)}_{\alpha}}{\partial t}, w-u^{(n)}_\alpha)_{\mu_n}dt},\label{all}\\
\int_{t_1}^{t_2}{(f_{\alpha,n},w_{R_i}-u^{(n)}_{\alpha,R_i})_{\mu_n} dt} &\to& \int_{t_1}^{t_2}{(f_{\alpha,n},w-u^{(n)}_{\alpha})_{\mu_n} dt},\label{all-b}\\
\int_{t_1}^{t_2}a^{(\alpha,n)}_\mu(u^{(n)}_{\alpha,R_i},w_{R_i})dt &\to& \int_{t_1}^{t_2}a^{(\alpha,n)}_\mu(u^{(n)}_{\alpha},w)dt\,.\label{all-c}
\end{eqnarray}
As for the last term on the right hand side of \eqref{d-1}, consider
\begin{align}\label{all3}
\int_{t_1}^{t_2}{\hspace{-5pt}a^{(\alpha,n)}_\mu(u^{(n)}_{\alpha,R_i},u^{(n)}_{\alpha,R_i})dt}=&\hspace{-3pt}
\int_{t_1}^{t_2}{\hspace
{-5pt}a^{(\alpha,n)}_\mu(u^{(n)}_{\alpha,R_i}-u^{(n)}_{\alpha},u^{(n)}_{\alpha,R_i}-u^
{(n)}_{\alpha})dt}\\
&+\hspace{-3pt}\int_{t_1}^{t_2}{\hspace{-5pt}a^{(\alpha,n)}_\mu(u^{(n)}_{\alpha},u^{(n)}_{\alpha,R_i})dt}+
\hspace{-3pt}\int_{t_1}^{t_2}{\hspace{-5pt}a^{(\alpha,n)}_\mu(u^{(n)}_{\alpha,R_i}
-u^{(n)}_{\alpha},u^{(n)}_{\alpha})dt}.\nonumber
\end{align}
For the last two integrals argue as above, hence
\begin{eqnarray}\label{all-a1}
&&\lim_{i\to\infty}\int_{t_1}^{t_2}{a^{(\alpha,n)}_\mu(u^{(n)}_{\alpha},u^{(n)}_{\alpha,R_i})dt} = \int_{t_1}^{t_2}{a^{(\alpha,n)}_\mu(u^{(n)}_{\alpha},u^{(n)}_{\alpha})dt},
\end{eqnarray}
\begin{eqnarray}\label{all-a2}
&&\hspace{-60pt}\lim_{i\to\infty}\int_{t_1}^{t_2}{a^{(\alpha,n)}_\mu(u^{(n)}_{\alpha,R_i}-u^{(n)}_{\alpha},u^{(n)}
_{\alpha})dt} = 0.
\end{eqnarray}
On the other hand, to the first integral in \eqref{all3} apply arguments similar to those in the proof of Theorem \ref{uni-est} to get
\begin{align}\label{all4}
\int_{t_1}^{t_2}\hspace{-5pt}a^{(\alpha,n)}_\mu &(u^{(n)}_{\alpha,R_i}-u^{(n)}_{\alpha},u^{(n)}_{\alpha,R_i}-u^{(n)}_{\alpha})dt\\
&\geq -C_p\Big\|Du^{(n)}_{\alpha,R_i}-Du^{(n)}_{\alpha}\Big\|_{L^2(0,T;L^{2p'}(\mathbb{R}^n,\mu_n;\mathbb{R}^n))}
\Big\|u^{(n)}_{\alpha,R_i}-u^{(n)}_{\alpha}\Big\|_{L^2(0,T;L^{2p}(\mathbb{R}^n,\mu_n))},\nonumber
\end{align}
with $p$ and $p'$ as in \eqref{b-sp1} and $C_p>0$ a suitable constant independent of $i$, $\alpha$ and $n$.
It then follows from Proposition \ref{univ-est-sol} and Lemma \ref{weak01} that
\begin{eqnarray}\label{all-a3}
\lim_{i\to\infty}\int_{t_1}^{t_2}{a^{(\alpha,n)}_\mu(u^{(n)}_{\alpha,R_i}-u^{(n)}_{\alpha},u^{(n)}_{\alpha,R_i}-u^{(n)}_
{\alpha})dt}\geq0.
\end{eqnarray}
Now \eqref{all-a1}, \eqref{all-a2} and \eqref{all-a3} imply
\begin{eqnarray}\label{all5}
\lim_{i\to\infty}\int_{t_1}^{t_2}{a^{(\alpha,n)}_\mu(u^{(n)}_{\alpha,R_i},u^{(n)}_{\alpha,R_i})dt}\geq \int_{t_1}^{t_2}{a^{(\alpha,n)}_\mu(u^{(n)}_{\alpha},u^{(n)}_{\alpha})dt}.
\end{eqnarray}
Therefore  \eqref{d-1}, \eqref{all}, \eqref{all-b}, \eqref{all-c}, \eqref{all5} show the convergence of \eqref{str-f} to \eqref{d-2} since $t_1$ and $t_2$ are arbitrary.

The continuity of $u^{(n)}_\alpha$ follows from Proposition \ref{unb-d} and Corollary \ref{str-sol-obs}.
As for the optimality of $\tau^\star_{\alpha,n}(t,x)$, notice that its proof is a simpler version of the one of Lemma \ref{tdar} and Theorem \ref{os-t1} below, hence it is only outlined here. For any initial data $(t,x)$ one has
\begin{equation}\label{brabba}
\lim_{R\to\infty}\tau^\star_{\alpha,n,R}(t,x)\wedge\tau^\star_{\alpha,n}(t,x)=\tau^\star_{\alpha,n}(t,x)
\quad\mathbb{P}\text{-a.s.}
\end{equation}
by an extension of \cite{Ben-Lio82}, Chapter 3, Section 3, Theorem 3.7 and by our Proposition \ref{unb-d}. Since $\tau^\star_{\alpha,n,R}$ is optimal for $\mathcal{U}^{(n)}_{\alpha,R}$ and $\tau^\star_{\alpha,n,R}\wedge\tau^\star_{\alpha,n}\leq\tau^\star_{\alpha,n,R}$ $\mathbb{P}$-a.s., it follows from \eqref{dyn-prg-3} that
\begin{eqnarray}\label{dyn-prg-4}
\mathcal{U}^{(n)}_{\alpha,R}(t,x^{(n)})=\mathbb{E}\left\{\mathcal{U}^{(n)}_{\alpha,R}(\tau^\star_
{\alpha,n,R}\wedge\tau^\star_{\alpha,n},X^{(\alpha)t,x;n}_{\tau^\star_{\alpha,n,R}
\wedge\tau^\star_{\alpha,n}})\right\}.
\end{eqnarray}
Therefore, Proposition \ref{unb-d}, the continuity of $\mathcal{U}^{(n)}_{\alpha}$ and \eqref{brabba} provide
\begin{eqnarray}\label{basev2}
\mathcal{U}^{(n)}_{\alpha}(t,x^{(n)})=\mathbb{E}\left\{\mathcal{U}^{(n)}_{\alpha}(\tau^\star_{\alpha,n},X^{(\alpha)t,x;n}_
{\tau^\star_{\alpha,n}})\right\}=\mathbb{E}\left\{\Theta^{(n)}(\tau^\star_{\alpha,n},X^{(\alpha)t,x;n}_
{\tau^\star_{\alpha,n}})\right\}
\end{eqnarray}
by taking limits as $R\to\infty$ in \eqref{dyn-prg-4}. It follows that $\tau^\star_{\alpha,n}$ is optimal.
\end{proof}
\begin{remark}\label{orrore}
Notice that for any stopping time $\sigma$ the same arguments that provide \eqref{brabba} also give
\begin{equation}
\lim_{R\to\infty}\tau^\star_{\alpha,n,R}\wedge\tau^\star_{\alpha,n}\wedge\sigma
=\tau^\star_{\alpha,n}\wedge\sigma\qquad\mathbb{P}\text{-a.s.}
\end{equation}
Therefore one has
\begin{equation}\label{BB}
\mathcal{U}^{(n)}_\alpha(t,x^{(n)})=\mathbb{E}\bigg\{\mathcal{U}^{(n)}_\alpha(\sigma,X^{(\alpha)t,x;n}_\sigma)\bigg\}\qquad
\text{for $\sigma\leq\tau^\star_{\alpha,n}$, $\mathbb{P}$-a.s.}
\end{equation}
\end{remark}

\subsection{Infinite dimensional domains}\label{sec-infty}
\subsubsection{The variational inequality for bounded operator $A_\alpha$}\label{refer}
Define the infinite-dimensional counterpart of $\mathcal{V}^p_n$ of Definition \ref{b-sp} by setting
\begin{equation}\label{def-sp-V}
\mathcal{V}^p:=\{v\,:\,v\in L^{2p}(\mathcal{H},\mu)\:\textit{and}\:\:D v\in L^{2p'}(\mathcal{H},\mu;\mathcal{H})\}.
\end{equation}
Endow $\mathcal{V}^p$ with the norm
\begin{eqnarray}
\lnr v\rnr_{p}:=\|v\|_{L^{2p}(\mathcal{H},\mu)}+\|D v\|_{L^{2p'}(\mathcal{H},\mu;\mathcal{H})}
\end{eqnarray}
so to obtain a separable Banach space. Notice that $\mathcal{V}^p_n\subset\mathcal{V}^p$ by Remark \ref{conn}. Also, by \eqref{univ}
\begin{equation}\label{univ-2}
\int_0^T{|a^{(\alpha,n)}_\mu(u(t),w(t))| dt}\leq C_{\mu,\gamma,p}\left(\int_0^T{\lnr u(t)\rnr^2_{p}}dt\right)^\frac{1}{2}\left(\int_0^T{\lnr w(t)\rnr^2_{p}}dt\right)^\frac{1}{2}
\end{equation}
for $u,w\in L^2(0,T;\mathcal{V}^p)$.

Denote by $\mathcal{L}_\alpha$ the infinitesimal generator of $X^{(\alpha)}$ (cf.\ \eqref{SDE-yos}); that is,
\begin{equation}\label{ellalpha}
\mathcal{L}_\alpha\, g(x)=\frac{1}{2}Tr\left[\sigma\sigma^*(x)D^2g(x)\right]+\langle A_\alpha x, Dg(x)\rangle\:\:\:\:\textit{for}\,\,g\in C^2_b(\mathcal{H}).
\end{equation}
The bilinear form associated to \eqref{ellalpha} is the infinite-dimensional counterpart of \eqref{lup2} and it is given by
\begin{align}\label{lup3}
a^{(\alpha)}_{\mu}&(u,w):=\int_{\mathcal{H}}\frac{1}{2}\langle B\,Du,D w\rangle _\mathcal{H}\,\mu(dx)+\hspace{-4pt}\int_{\mathcal{H}}\langle\, \overline{C}^{(\alpha)},D u\rangle_{\mathcal{H}}\, w\,\mu(dx)
\end{align}
with $B:=\sigma\sigma^{*}$, $\overline{C}^{(\alpha)}=\frac{1}{2}\left(Tr[D\sigma]_{\mathcal{H}}\sigma+D\sigma\cdot\sigma-
2A_{\alpha}x-\sigma\sigma^{*}Q^{-1}x\right)$ and $D\sigma\cdot\sigma$ denotes the action of $D\sigma\in\mathcal{L}(\mathcal{H})$ on $\sigma\in\mathcal{H}$.

Let $w\in L^2(0,T;\mathcal{V}^p)$ and $(w_n)_{n\in\mathbb{N}}\subset L^2(0,T;\mathcal{V}^p)$ be such that $w_n\to w$. Then, for arbitrary $0\le t_1<t_2\le T$, define $\mathcal{T}_{\alpha,w}(t_1,t_2)\in L^2(0,T;\mathcal{V}^{p})^*$ and the sequence $(\mathcal{T}^{\,n}_{\alpha,w}(t_1,t_2))_{n\in\mathbb{N}}\subset L^2(0,T;\mathcal{V}^{p})^*$ by setting
\begin{align}\label{b-form-e}
\mathcal{T}_{\alpha,w}(t_1,t_2)(\,\cdot\,):=\int_{t_1}^{t_2}{a^{(\alpha)}_{\mu}(\,\cdot\,,w)dt}\quad \textrm{and}\quad\mathcal{T}^{\,n}_{\alpha,w}(t_1,t_2)(\,\cdot\,):=\int_{t_1}^{t_2}{a^{(\alpha,n)}_{\mu}(\,\cdot\,,w_n)dt}.
\end{align}
Tedious but straightforward calculations give
\begin{equation}\label{b-form-c}
\lim_{n\to\infty}\|\mathcal{T}^{\,n}_{\alpha,w}(t_1,t_2)-\mathcal{T}_{\alpha,w}(t_1,t_2)\|_{L^2(0,T;\mathcal{V}^{p})^*}=0.
\end{equation}
Also, recall $f_{\alpha,n}$ of \eqref{gen-inf} and set
\begin{align}\label{gen-inf2}
f_\alpha:=\frac{\partial\Theta}{\partial t}+\mathcal{L}_\alpha\Theta;
\end{align}
then it holds
\begin{equation}\label{b-form-d}
\lim_{n\to\infty}\int_0^T\|f_{\alpha,n}-f_{\alpha}\|^2_{L^p(\mathcal{H},\mu)}\,dt=0,\qquad1\leq p<\infty
\end{equation}
by Assumptions \ref{ass-sigma} and \ref{ass-psi} and dominated convergence theorem. Finally, similarly to $\mathcal{K}^{p}_{n,\mu}$ of \eqref{ccs}, for $1< p< \infty$ define the closed, convex set
\begin{equation}\label{ccs2}
\mathcal{K}^p_{\mu}:=\Big\{w:\,w\in \mathcal{V}^p\:\:\text{and}\:\:w\geq0\:\mu\text{-a.e.}\Big\}.
\end{equation}
\begin{lemma}\label{t-func}
Let $w\in\mathcal{K}^p_{\mu}$ for some $1<p<+\infty$. Then there exists a double-indexed sequence $ (w_{k,n})_{k,n\in\mathbb{N}}\subset \mathcal{V}^p$ such that for $k$ fixed, $w_{k,n}\in\cap_{m\ge n}\mathcal{K}^p_{m,\mu}$.

Moreover,
\begin{eqnarray}
\lim_{k\to\infty}\lim_{n\to\infty}w_{k,n}=w\quad\text{weakly in $\mathcal{V}^p$ and strongly in $L^p(\mathcal{H},\mu)$,}
\end{eqnarray}
taking the limits in the prescribed order.
\end{lemma}
\begin{proof}
Since $D(A^*)$ is dense in $\mathcal{H}$ the set
\begin{equation}\label{dense}
\hspace{-5pt}\mathcal{E}_A(\mathcal{H}):=\textrm{span}\Big\{\mathscr{R}e(\varphi_{h}),\,\mathscr{I}m(\varphi_{h}),
\:\varphi_{h}(x)=e^{i\langle h,x\rangle_{\mathcal{H}}},\,h\in D(A^*)\Big\}
\end{equation}
is dense\footnote{The proof relies on the fact that the set of continuous functions is dense in $L^p(\mathcal{H},\mu)$ and goes through a finite-dimensional reduction, a localization and the Stone-Weierstrass theorem.} in $\mathcal{V}^p$ (cf.~\cite{DaPr}, Chapter 10 and \cite{DaPr-Zab04}, Chapter 9).  Hence for $w\in\mathcal{K}^p_{\mu}$ there exists a sequence $(\phi^{(k)})_{k\in\mathbb{N}}\subset\mathcal{E}_A(\mathcal{H})$ such that $\phi^{(k)}\to w$ in $\mathcal{V}^p$ as $k\to\infty$. Recall the projection $P_n$ and set $\phi^{(k)}_n(x):=\phi^{(k)}(P_nx)$ for $n\in\mathbb{N}$. Since $\phi^{(k)}$ is a finite linear combination of elements in $\mathcal{E}_A(\mathcal{H})$ and it is continuous and bounded alongside with $D\phi^{(k)}$, dominated convergence implies $\phi^{(k)}_n\to\phi^{(k)}$ in $\mathcal{V}^p$ as $n\to\infty$. It follows that $(\phi^{(k)}_n)_{k,n\in\mathbb{N}}$ is bounded in $\mathcal{V}^p$ and so is $(\phi^{(k)}_{n,0})_{k,n\in\mathbb{N}}$ where $\phi^{(k)}_{n,0}:=0\vee\phi^{(k)}_n=[\phi^{(k)}_n]^+$. Therefore by taking limits as $n\to\infty$ first, and as $k\to\infty$ afterwards, one obtains weak convergence in $\mathcal{V}^p$ of $\phi^{(k)}_{n,0}$ to some function $g$. However, $\big|\phi^{(k)}_{n,0}-w\big|=\big| [\phi^{(k)}_{n}]^+-[w]^+\big|\le\big|\phi^{(k)}_{n}-w\big|$ for all $x\in\mathcal{H}$, since $w\ge0$. Therefore dominated convergence implies $\phi^{(k)}_{n,0}\to w$ in $L^{p}(\mathcal{H},\mu)$ as limits are taken in the same order as before and we may conclude $g\equiv w$. Clearly, for $k$ fixed, $\phi^{(k)}_{n,0}\in\cap_{m\ge n}\mathcal{K}^p_{m,\mu}$ and the Lemma follows by setting $w_{k,n}:=\phi^{(k)}_{n,0}$.
\end{proof}
Recall the value function $\mathcal{U}_\alpha$ of the optimal stopping problem \eqref{OS2} and set $u_\alpha:=\mathcal{U}_{\alpha}-\Theta$. Then Assumption \ref{ass-psi}, Theorem \ref{cnv2} and the same bounds as those employed to obtain Lemma \ref{weak01} provide the following
\begin{lemma}\label{weak02}
There exists a sequence $(n_i)_{i\in\mathbb{N}}$ such that $n_i\to\infty$ as $i\to \infty$ and
$u^{(n_i)}_{\alpha}$ converges to $u_\alpha$ as $n_i\to\infty$, weakly in $L^p(0,T;\mathcal{V}^p)$ and strongly in $L^p(0,T; L^p(\mathcal{H},\mu))$, $1\leq p<\infty$.

Moreover, $\frac{\partial\,}{\partial\,t}u^{(n_i)}_{\alpha}$ converges to $\frac{\partial\,}{\partial\,t}u_{\alpha}$ as $n_i\to\infty$ weakly in $L^2(0,T;L^2(\mathcal{H},\mu))$.
\end{lemma}
Denote by $(\cdot,\cdot)_\mu$ the scalar product in $L^2(\mathcal{H},\mu)$.
\begin{theorem}\label{infty-dim-vi}
For every $1<p<\infty$ the function $u_\alpha$ is a solution of the variational problem on $\mathcal{H}$
\begin{eqnarray}\label{d-2bis}
\left\{
\begin{array}{l}
u\in L^2(0,T;\mathcal{V}^p);\:\:\:\:\frac{\partial\,u}{\partial\,t}\in L^2(0,T;L^2(\mathcal{H},\mu));\\
\\
u(T,x)=0,\:\text{for $x\in\mathcal{H};$}\:\:\:\:u(t,x)\geq 0,\:\text{for $(t,x)\in[0,T]\times\mathcal{H}$;}\\
\\
\hspace{-3pt}\displaystyle{-\big(\frac{\partial u}{\partial t}(t),w-u(t)\big)_{\mu}+a^{(\alpha)}_\mu\big(u(t),w-u(t)\big)- \big(f_\alpha(t),w-u(t)\big)_{\mu}\geq0},\\
\hspace{+230pt}\text{for a.e.~$t\in[0,T]$ and for all $w\in \mathcal{K}^p_{\mu}$.}
\end{array}\right.
\end{eqnarray}
Moreover, $u_\alpha\in C([0,T]\times\mathcal{H})$.
\end{theorem}
\begin{proof}
The continuity of $u_\alpha$ is a consequence of Corollary \ref{cont1} and Proposition \ref{unb-d}. For arbitrary $w\in\mathcal{K}^p_{\mu}$ take the corresponding approximating sequence $(w_{k,n})_{k,n\in\mathbb{N}}$ given by Lemma \ref{t-func}. For $k\in\mathbb{N}$ arbitrary but fixed, Theorems \ref{unb-dom}, Lemma \ref{t-func} and Remark \ref{conn} guarantee
\begin{align*}
-\big(\frac{\partial u^{(m)}_\alpha}{\partial t}(t),w_{k,n}-u^{(m)}_\alpha(t)\big)_{\mu}+ a_{\mu}^{(\alpha,m)}(u^{(m)}_\alpha(t),w_{k,n}-u^{(m)}_\alpha(t))-(f_{\alpha,m}(t),w_{k,n}-u^{(m)}
_\alpha(t))_\mu\geq0,
\end{align*}
for $m\geq n$ and a.e.~$t\in[0,T]$. In the limit as $m\to\infty$, Lemma \ref{weak02}, equations \eqref{b-form-c} and \eqref{b-form-d} and arguments similar to those used in the proof of Theorem \ref{unb-dom} give
\begin{align*}
\int_{t_1}^{t_2}\hspace{-2pt}\Big[\hspace{-2pt}-\big(\frac{\partial u_\alpha}{\partial t}(t),w_{k,n}-u_\alpha(t)\big)_{\mu}\hspace{-2pt}+a_{\mu}^{(\alpha)}(u_\alpha(t),w_{k,n}-u_\alpha(t))-
(f_\alpha(t),w_{k,n}-u_\alpha(t))_\mu\Big]dt\hspace{-2pt}\geq0.
\end{align*}
The proof now follows from Theorem \ref{t-func} by taking limits as $n,\,k\to\infty$ and then dividing by $t_2-t_1$ and letting $t_2-t_1\to0$.
\end{proof}
The existence of an optimal stopping time for $\mathcal{U}_\alpha$ of \eqref{OS2} is obtained by purely probabilistic considerations (cf.~Theorem \ref{os-t1} below). Two preliminary lemmas are needed. Given $(t,x)\in[0,T]\times\mathcal{H}$, let $\tau^\star_{\alpha,n}(t,x)$ be as in \eqref{bis} and define
\begin{equation}\label{tau}
\tau^\star_{\alpha}(t,x):=\inf\{s\geq t\,:\,\mathcal{U}_\alpha(s,X^{(\alpha)t,x}_s)=\Theta(s,X^{(\alpha)t,x}_s)\}\wedge T.
\end{equation}
\begin{lemma}\label{tdar}
{For each $(t,x)\in[0,T]\times\mathcal{H}$} there exists a subsequence $(\tau^\star_{\alpha,n_j}(t,x))_{j\in\mathbb{N}}$, with $n_j=n_j(t,x)$, such that $n_j\to\infty$ as $j\to\infty$ and
\begin{equation}\label{lim2}
\lim_{j\to\infty}(\tau^\star_{\alpha}(t,x)\wedge\tau^\star_{\alpha,n_j}(t,x))(\omega)=\tau^\star_{\alpha}(t,x)
(\omega),\quad\mathbb{P}\textrm{-a.e.}\:\omega\in\Omega.
\end{equation}
\end{lemma}
\begin{proof}
Fix $x_0\in \mathcal{H}$. There is no loss of generality if we consider the diffusions $X^{(\alpha)x_0}$ and $X^{(\alpha)x_0;n}$ starting at time zero as all results remain true for arbitrary initial time $t$. The proof of this Lemma is adapted from \cite{Ben-Lio82}, Chapter 3, Section 3, Theorem 3.7 (cf.~in particular p. 322).

Using Proposition \ref{c-1}, fix $\Omega_0\subset\Omega$ with $\mathbb{P}(\Omega_0)=1$ and a subsequence $(n_j)_{j\in\mathbb{N}}$, with $n_j=n_j(x_0)$, such that
\begin{eqnarray}\label{subseq}
\lim_{j\to\infty}\sup_{0\leq t\leq T}\left\|X^{(\alpha)x_0;n_j}_t(\omega)-X^{(\alpha)x_0}_t(\omega)\right\|_{\mathcal{H}}\to0,\quad\text{for all $\omega\in\Omega_0$.}
\end{eqnarray}
Since the starting point $x_0\in\mathcal{H}$ is fixed, to simplify the notation in the rest of the proof, we shall write $\tau^\star_{\alpha,n}$ and $\tau^\star_{\alpha}$ instead of $\tau^\star_{\alpha,n}(0,x_0)$ and $\tau^\star_{\alpha}(0,x_0)$, respectively.
The limit \eqref{lim2} is trivial if $\omega^\prime\in\Omega_0$ is such that $\tau^\star_\alpha(\omega^\prime)=0$. On the other hand, if $\omega^\prime\in\Omega_0$ is such that $\tau^\star_\alpha(\omega')>\delta$ for some $\delta=\delta_{x_0}>0$, then by \eqref{tau}
\begin{eqnarray*}
\mathcal{U}_\alpha(t,X^{(\alpha)x_0}_t(\omega^\prime))>\Theta(t,X^{(\alpha)x_0}_t(\omega^\prime)),\qquad t\in[0,\tau^\star_\alpha(\omega^\prime)-\delta].
\end{eqnarray*}
Since the map $t\mapsto X^{(\alpha)x_0}_t(\omega^\prime)$ is continuous and $[0,\tau^\star_\alpha(\omega^\prime)-\delta]$ is a compact set it follows that the set $\chi^\delta(\omega^\prime):=\{y\in\mathcal{H}:\,y=X^{(\alpha)x_0}_t(\omega^\prime)\,,\,t\in[0,\tau^\star_\alpha(\omega^\prime)-\delta]\}$ is a compact subset of $\mathcal{H}$. Therefore the continuous map $(t,x)\mapsto \mathcal{U}_\alpha(t,x)-\Theta(t,x)$ (cf.~Theorem \ref{infty-dim-vi}) attains its minimum on $[0,\tau^\star_\alpha(\omega)-\delta]\times\chi^\delta(\omega^\prime)$, call it $\rho(\delta,\omega^\prime)>0$. Then
\begin{eqnarray}\label{nu1}
\mathcal{U}_\alpha(t, X^{(\alpha)x_0}_t(\omega^\prime))\geq\Theta(t, X^{(\alpha)x_0}_t(\omega^\prime))+\rho(\delta,\omega^\prime),\qquad t\in[0,\tau^\star_\alpha(\omega^\prime)-\delta].
\end{eqnarray}
Recall from Theorem \ref{cnv2} and \eqref{psi-n2} that $\mathcal{U}^{(n)}_\alpha$ and $\Theta^{(n)}$ converge respectively to $\mathcal{U}_\alpha$ and $\Theta$, uniformly on compact subsets of $[0,T]\times\mathcal{H}$. Therefore there exists $n_\rho=n(\rho(\delta,\omega^\prime))\in(n_j)_{j\in\mathbb{N}}$, $n_\rho>0$ large enough such that
\begin{eqnarray}
&&\mathcal{U}^{(n_\rho)}_\alpha(t,y^{(n_\rho)})> \mathcal{U}_\alpha(t,y)-\frac{1}{4}\rho(\delta,\omega^\prime),\qquad(t,y)\in[0,\tau^\star_\alpha(\omega^\prime)-\delta]
\times\chi^\delta(\omega^\prime),\label{nu2}\\
&&\Theta^{(n_\rho)}(t,y^{(n_\rho)})< \Theta(t,y)+\frac{1}{4}\rho(\delta,\omega^\prime),\qquad(t,y)
\in[0,\tau^\star_\alpha(\omega^\prime)-\delta]\times\chi^\delta(\omega^\prime),\label{nu3}
\end{eqnarray}
and
\begin{equation}\label{nu4}
\sup_{0\leq t\leq T}\big\|X^{(\alpha)x_0;n_\rho}(\omega^\prime)-X^{(\alpha)x_0}(\omega^\prime)\big\|_{\mathcal{H}}
\leq\frac{1}{4L_\mathcal{U}\vee L_\Theta}\rho(\delta,\omega^\prime).
\end{equation}
Now \eqref{nu1}, \eqref{nu2} and \eqref{nu3} imply
\begin{eqnarray}\label{cane}
\mathcal{U}^{(n_\rho)}_\alpha(t,P_{n_\rho}X^{(\alpha)x_0}_t(\omega^\prime))> \Theta^{(n_\rho)}(t,P_{n_\rho}X^{(\alpha)x_0}_t(\omega^\prime))+\frac{1}{2}\rho(\delta,\omega^\prime),\qquad t\in[0,\tau^\star_\alpha(\omega^\prime)-\delta].
\end{eqnarray}
On the other hand Assumption \ref{ass-psi}, Proposition \ref{reg-v-fun} and the fact that $P_{n_\rho}X^{(\alpha)x_0;n_\rho}=X^{(\alpha)x_0;n_\rho}$ imply
\begin{align}
\hspace{-2pt}\sup_{0\leq t\leq T}\hspace{-1pt}\left|\mathcal{U}^{(n_\rho)}_\alpha(t,P_{n_\rho}X^{(\alpha)x_0}_t(\omega^\prime))-\mathcal{U}^{(n_\rho)}
_\alpha(t,X^{(\alpha)x_0;n_\rho}_t(\omega^\prime))\right|\hspace{-2pt}\leq\hspace{-2pt} L_\mathcal{U}\hspace{-3pt}\sup_{0\leq t\leq T}\hspace{-1pt}\left\|X^{(\alpha)x_0;n_\rho}_t(\omega^\prime)\hspace{-2pt}-\hspace{-2pt}X^{(\alpha)x_0}_t
(\omega^\prime)\right\|_{\mathcal{H}}
\label{cane2}
\end{align}
and
\begin{align}
\hspace{-2pt}\sup_{0\leq t\leq T}\hspace{-1pt}\left|\Theta^{(n_\rho)}(t,P_{n_\rho}X^{(\alpha)x_0}_t(\omega^\prime))-\Theta^{(n_\rho)}(t,X^{(\alpha)
x_0;n_\rho}_t(\omega^\prime))\right|\hspace{-2pt}\leq \hspace{-2pt}L_{\Theta}\hspace{-3pt}\sup_{0\leq t\leq T}\hspace{-1pt}\left\|X^{(\alpha)x_0;n_\rho}_t(\omega^\prime)\hspace{-2pt}-\hspace{-2pt}X^{(\alpha)x_0}_t
(\omega^\prime)\right\|
_{\mathcal{H}}\label{cane3}
\end{align}
which, together with \eqref{nu4} and \eqref{cane}, imply
\begin{eqnarray*}
\mathcal{U}^{(n_\rho)}_\alpha(t,X^{(\alpha)x_0;n_\rho}_t(\omega^\prime))> \Theta^{(n_\rho)}(t,X^{(\alpha)x_0;n_\rho}_t(\omega^\prime)),\qquad t\in[0,\tau^\star_\alpha(\omega^\prime)-\delta].
\end{eqnarray*}
It follows that $\tau^\star_{\alpha,n_\rho}(\omega^\prime)>\tau^\star_{\alpha}(\omega^\prime)-\delta$. Notice that $\rho(\delta,\omega^\prime)\to0$ as $\delta\to0$ and hence $n_\rho\to\infty$. Therefore $\tau^\star_{\alpha,n_\rho}(\omega^\prime)\wedge\tau^\star_{\alpha}(\omega^\prime)\to\tau^\star_{\alpha}(\omega^\prime)$ as $n_\rho\to\infty$, which is equivalent to say that \eqref{lim2} holds along a subsequence.
\end{proof}
Notice that arguments as in the proof of \eqref{lip-V} also give
\begin{align}\label{lip-al}
\sup_{0\le t\le T}\big|\mathcal{U}^{(n)}_\alpha(t,x^{(n)})-\mathcal{U}_\alpha(t,x)\big|\le L_{\mathcal{U}}\|x-x^{(n)}\|_{\mathcal{H}},
\end{align}
since the optimal stopping problems \eqref{OS1}, \eqref{OS2} and \eqref{OS3} are considered under the same filtration $\{\mathcal{F}_t,\,t\ge0\}$.
\begin{theorem}\label{os-t1}
The optimal stopping time of \eqref{OS2} is $\tau^\star_{\alpha}(t,x)$ as defined in \eqref{tau}. Moreover
\begin{align}\label{ualpha-dyn}
\mathcal{U}_\alpha(t,x)=\mathbb{E}
\left\{\mathcal{U}_\alpha(\sigma\wedge\tau_\alpha^\star,X^{(\alpha)t,x}_{\sigma\wedge\tau_\alpha^\star})
\right\}\quad\text{for all stopping times $t\le\sigma\le T$.}
\end{align}
\end{theorem}
\begin{proof}
Given the initial data $(t,x)$, we adopt the simplified notation used in the proof of Lemma \ref{tdar}; that is, we set $\tau^\star_\alpha:=\tau^\star_{\alpha}(t,x)$ and $\tau^\star_{\alpha,n}:=\tau^\star_{\alpha,n}(t,x)$. By Remark \ref{orrore} we have
\begin{eqnarray}\label{vf1}
\mathcal{U}^{(n)}_\alpha(t,x^{(n)})=\mathbb{E}\left\{\mathcal{U}^{(n)}_\alpha(\tau^\star_{\alpha}\wedge\tau^\star_
{\alpha,n},X^{(\alpha)t,x;n}_{\tau^\star_{\alpha}\wedge\tau^\star_{\alpha,n}})\right\}.
\end{eqnarray}
In \eqref{vf1} take the subsequence $(n_j)_{j\in\mathbb{N}}$ of Lemma \ref{tdar} and apply Theorem \ref{cnv2} to obtain the convergence of $\mathcal{U}^{(n_j)}_\alpha(t,x^{(n_j)})$ to $\mathcal{U}_\alpha(t,x)$ as $j\to\infty$.

On the other hand
\begin{align}\label{brad}
\bigg|\mathbb{E}\Big\{\mathcal{U}^{(n_j)}_\alpha(\tau^\star_{\alpha}\wedge\tau^\star_{\alpha,n_j},
X^{(\alpha)t,x;n_j}_{\tau^\star_{\alpha}\wedge\tau^\star_{\alpha,n_j}})&-\mathcal{U}_\alpha(\tau^\star_{\alpha},X^
{(\alpha)t,x}_{\tau^\star_{\alpha}})\Big\}\bigg|\nonumber\\
&\leq\mathbb{E}\bigg\{\sup_{t\leq s\leq T}\bigg|\mathcal{U}^{(n_j)}_\alpha(s,X^{(\alpha)t,x;n_j}_s)-\mathcal{U}^{(n_j)}_\alpha(s,P_{n_j}X^{(\alpha)t,x}_s)
\bigg|\bigg\}\nonumber\\
&+\mathbb{E}\bigg\{\sup_{t\leq s\leq T}\bigg|\mathcal{U}^{(n_j)}_\alpha(s,P_{n_j}X^{(\alpha)t,x}_s)-\mathcal{U}_\alpha(s,X^{(\alpha)t,x}_s)\bigg|\bigg\}
\\
&+\mathbb{E}\bigg\{\bigg|\mathcal{U}_\alpha(\tau^\star_{\alpha}\wedge\tau^\star_{\alpha,n_j},X^{(\alpha)t,x}_{\tau^
\star_{\alpha}\wedge\tau^\star_{\alpha,n_j}})-\mathcal{U}_\alpha(\tau^\star_{\alpha},X^{(\alpha)t,x}_
{\tau^\star_{\alpha}})\bigg|\bigg\},\nonumber
\end{align}
where the first term on the right hand side goes to zero as $j\to\infty$ by \eqref{lip-V}, Proposition \ref{c-1} and Jensen's inequality. Similarly, the second term goes to zero by \eqref{lip-al} and dominated convergence, and the third term goes to zero by dominated convergence and Lemma \ref{tdar}.

In conclusion, by taking the limits in \eqref{vf1} along the subsequence $(n_j)_{j\in\mathbb{N}}$ we obtain
\begin{equation}\label{last}
\mathcal{U}_\alpha(t,x)=\mathbb{E}\left\{\mathcal{U}_\alpha(\tau_\alpha^\star,X^{(\alpha)t,x}_{\tau_\alpha^\star})
\right\}
=\mathbb{E}\left\{\Theta(\tau_\alpha^\star,X^{(\alpha)t,x}_{\tau_\alpha^\star})\right\},
\end{equation}
and the optimality of $\tau_\alpha^\star$ follows. Similar arguments are used to prove \eqref{ualpha-dyn} since Lemma \ref{tdar} implies  $\sigma\wedge\tau^\star_{\alpha}\wedge\tau^\star_{\alpha,n_j}\to\sigma\wedge\tau^\star_{\alpha}$ as $j\to\infty$.
\end{proof}


\subsubsection{Removal of the Yosida approximation}\label{sec-yos}

The function $u_\alpha$ in Theorem \ref{infty-dim-vi} solves the variational inequality associated to the Yosida approximation $A_\alpha$ of the unbounded operator $A$. In this section we study the limiting behavior, as $\alpha\to\infty$, of $u_\alpha$ and of the corresponding variational inequality by adopting both pro\-ba\-bi\-li\-stic and analytical tools.

When $\alpha\to\infty$ the term involving $A_\alpha$ in the bilinear form $a^{(\alpha)}_\mu(\,\cdot\,,\,\cdot\,)$ of \eqref{lup3} converges to a suitable operator that needs to be fully characterized. Let $w\in \mathcal{V}^p$ be given and define the linear functional $L^{(\alpha)}_{A}(w,\cdot)\in \mathcal{V}^{p\,*}$ by
\begin{equation}\label{rab1}
L^{(\alpha)}_{A}(w,u):=\int_{\mathcal{H}}\langle A_\alpha x,Du\rangle_{\mathcal{H}}\, w\, \mu(dx),\qquad u\in \mathcal{V}^p.
\end{equation}
It is easy to show that $L^{(\alpha)}_{A}(w,\cdot)$ is continuous by \eqref{stimaA} and any sequence $(L^{(\alpha_n)}_{A}(w,\cdot))_{n\in\mathbb{N}}$, with $\alpha_n\to\infty$ as $n\to\infty$, is a Cauchy sequence in $\mathcal{V}^{p\,*}$. In fact for $n>m$ arguments similar to those in \eqref{stimaA} give
\begin{align*}
|L^{(\alpha_n)}_{A}(w,u)-L^{(\alpha_m)}_{A}(w,u)|&=\left|\int_{\mathcal{H}}\langle (A_{\alpha_n}-A_{\alpha_m})x,Du\rangle_{\mathcal{H}}w\,\mu(dx)\right|\\
&\leq C_p\,Tr\big[(A_{\alpha_n}-A_{\alpha_m})Q(A_{\alpha_n}-A_{\alpha_m})^*\big]\,\lnr u\rnr_p\,\lnr w\rnr_p,
\end{align*}
and hence
\begin{eqnarray}\label{rab2}
\|L^{(\alpha_n)}_{A}(w,\,\cdot\,)-L^{(\alpha_m)}_{A}(w,\,\cdot\,)\|_{\mathcal{V}^{p\,*}}\leq C_p\,Tr\big[(A_{\alpha_n}-A_{\alpha_m})Q(A_{\alpha_n}-A_{\alpha_m})^*\big]\,\lnr w\rnr_p.
\end{eqnarray}
Since $A_\alpha\to A$ on $D(A)$ as $\alpha\to\infty$ and Assumption \ref{ass-Q} holds, \eqref{rab2} goes to zero as $m,n\to\infty$ and $(L^{(\alpha_n)}_{A}(w,\,\cdot\,))_{n\in\mathbb{N}}$ is Cauchy in $\mathcal{V}^{p\,*}$. Therefore, by completeness of $\mathcal{V}^{p\,*}$ there exists $\hat{L}_{A}(w,\,\cdot\,)\in \mathcal{V}^{p\,*}$ such that $L^{(\alpha)}_{A}(w,\,\cdot\,)\to \hat{L}_{A}(w,\cdot)$ as $\alpha\to\infty$ in $\mathcal{V}^{p\,*}$.

It suffices to characterize $\hat{L}_{A}(w,\cdot)$ on the set $\mathcal{E}_A(\mathcal{H})$ of \eqref{dense} since that is dense in $\mathcal{V}^p$. In order to do so notice that $A^*Du\in L^p(\mathcal{H},\mu)$ for $u\in\mathcal{E}_A(\mathcal{H})$ and \begin{align*}
\int_{\mathcal{H}}\langle A_\alpha x,Du\rangle_{\mathcal{H}}\, w \, \mu(dx)=\int_{\mathcal{H}}\langle x,A^*_\alpha Du\rangle_{\mathcal{H}}\, w \, \mu(dx).
\end{align*}
Now dominated convergence allows us to define a linear functional $L_{A}(w,\cdot)$ by setting
\begin{eqnarray}\label{rab3}
L_{A}(w,u):=\lim_{\alpha\to\infty}L^{(\alpha)}_{A}(w,u)=\int_{\mathcal{H}}\langle x,A^*D u\rangle_{\mathcal{H}}\, w \, \mu(dx),\quad\textrm{for}\:\:u\in\mathcal{E}_A(\mathcal{H}).
\end{eqnarray}
Clearly its domain $D(L_{A}(w,\cdot))$ contains $\mathcal{E}_A(\mathcal{H})$ and it is dense in $\mathcal{V}^p$. Since \eqref{stimaA} is uniform with respect to $n\in\mathbb{N}$ and $\alpha>0$ we also obtain
\begin{eqnarray}\label{ext}
|L_{A}(w,u)|\leq Tr\big[AQA^*\big] \lnr w\rnr_{p}\,\lnr u\rnr_p,\quad \textrm{for}\:\:u\in\mathcal{E}_A(\mathcal{H}).
\end{eqnarray}
By density arguments $L_{A}(w,\cdot)$ is continuously extended to the whole space $\mathcal{V}^p$ and the extended functional is denoted by $\bar{L}_{A}(w,\cdot)$.
It then follows
\begin{equation}\label{brut}
\hat{L}_{A}(w,\,\cdot\,):=\lim_{n\to\infty}L^{(\alpha_n)}_{A}(w,\,\cdot\,)=\bar{L}_{A}(w,\,\cdot\,)\qquad\textrm{in}
\:\:\mathcal{V}^{p\,*}.
\end{equation}

Note that, for $w\in L^2(0,T;\mathcal{V}^p)$ fixed, one has $\big(L^{(\alpha)}_A(w,\,\cdot\,)\big)_{\alpha>0}$ bounded in $L^{2}(0,T;\mathcal{V}^{p\,*})$ by \eqref{ext} (or by \eqref{stimaA}). Then for arbitrary $0\le t_1<t_2\le T$ and $u\in L^2(0,T;\mathcal{V}^p)$ we may define $T^{(\alpha)}_{A}(w,\cdot)(t_1,t_2)\in L^2(0,T;\mathcal{V}^{p})^*$ and $\bar{T}_{A}(w,\cdot)(t_1,t_2)\in L^2(0,T;\mathcal{V}^{p})^*$ by
\begin{equation}\label{rab4}
T^{(\alpha)}_{A}(w,u)(t_1,t_2):=\int_{t_1}^{t_2}{L^{(\alpha)}_A(w(t),u(t))dt},
\end{equation}
and
\begin{equation}\label{rab7}
\bar{T}_{A}(w,u)(t_1,t_2):=\int_{t_1}^{t_2}{\bar{L}_A(w(t),u(t))dt}.
\end{equation}
\begin{prop}\label{prot}
For arbitrary $0\le t_1<t_2\le T$, with $T^{(\alpha)}_{A}(w,\,\cdot\,)(t_1,t_2)$ and $\bar{T}_{A}(w,\,\cdot\,)(t_1,t_2)$ given by \eqref{rab4} and \eqref{rab7}, respectively, it holds that
\begin{equation}\label{onemore}
\lim_{\alpha\to\infty}\|(T^{(\alpha)}_{A}-\bar{T}_{A})(w,\,\cdot\,)(t_1,t_2)\|_{L^2(0,T;\mathcal{V}^{p})^*}=0.
\end{equation}
\end{prop}
\begin{proof}
A direct calculation gives $$\big|(T^{(\alpha)}_{A}-\bar{T}_{A})(w,u)(t_1,t_2)\big|\le\big\|(L^{(\alpha)}_{A}-\bar{L}_{A})(w,\,\cdot\,)
\big\|_{L^2(0,T;\mathcal{V}^{p\,*})}\big\|u\big\|
_{L^2(0,T;\mathcal{V}^{p\,*})}$$
and hence $\|(T^{(\alpha)}_{A}-\bar{T}_{A})(w,\,\cdot\,)(t_1,t_2)\|_{L^2(0,T;\mathcal{V}^{p})^*}\le
\big\|(L^{(\alpha)}_{A}-\bar{L}_{A})(w,\,\cdot\,)\big\|_{L^2(0,T;\mathcal{V}^{p\,*})}$. Now, since $\big\|(L^{(\alpha)}_{A}-\bar{L}_{A})(w(t),\,\cdot\,)
\big\|_{\mathcal{V}^{p\,*}}\le2 Tr[AQA^*]\lnr w(t)\rnr_p$ and the upper bound is independent of $\alpha$ and it belongs to $L^2(0,T)$, then dominated convergence theorem and \eqref{brut} give \eqref{onemore}.
\end{proof}

\begin{remark}\label{rotto}
Notice that for our gain function $\Theta$ we have $L^{(\alpha)}_{A}(\,\cdot\,,\Theta)\in L^2(0,T;\mathcal{V}^{p\,*})$. Moreover $T^{(\alpha)}_{A}(\,\cdot\,,\Theta)(t_1,t_2)\to \bar{L}_{A}(\,\cdot\,,\Theta)(t_1,t_2)$ in $L^2(0,T;\mathcal{V}^p)^*$ as $\alpha\to\infty$, for all $0\le t_1<t_2\le T$, by arguments similar to those used in the proof of Proposition \ref{prot}.
\end{remark}
\noindent For $t\in[0,T]$ define $F(\,\cdot\,)(t)\in \mathcal{V}^{p\,*}$ by
\begin{equation}\label{ruiz}
F(w)(t):=\Big(\frac{\partial\Theta}{\partial t}(t)+\frac{1}{2}Tr\big[\sigma\sigma^*D^2\Theta(t)\big],w\Big)_{\hspace{-2pt}\mu}+\bar{L}_{A}(w,\Theta(t)),\qquad\text{for all $w\in \mathcal{V}^p$}.
\end{equation}
Then, with $f_\alpha$ as in \eqref{gen-inf2}, from dominated convergence, Assumption \ref{ass-psi} and Remark \ref{rotto} follows that
\begin{equation}\label{prot2}
\lim_{\alpha\to\infty}\Big\|\int_{t_1}^{t_2}{\left[(f_\alpha(t),\,\cdot\,)_\mu-F(\cdot)(t)\right]dt}
\Big\|_{L^2(0,T;\mathcal{V}^{p})^*}=0
\end{equation}
for all $0\le t_1<t_2\le T$.

It is natural to consider the bilinear form associated to the infinitesimal generator of \eqref{SDE-infty},
\begin{align}\label{lup4}
a_{\mu}&(u,w):=\int_{\mathcal{H}}\frac{1}{2}\langle B\,Du,D w\rangle _\mathcal{H}\,\mu(dx)+\hspace{-4pt}\int_{\mathcal{H}}\hspace{-6pt}\langle\, \hat{C},D u\rangle_{\mathcal{H}}\, w\,\mu(dx)-\bar{L}_A(w,u)
\end{align}
for $u,w\in L^2(0,T;\mathcal{V}^p)$, and with $B$ as in \eqref{lup3} and $\hat{C}=\frac{1}{2}\left(Tr[D\sigma]_{\mathcal{H}}\sigma+D\sigma\cdot\sigma-\sigma\sigma^{*}Q^{-1}x\right)$. We set $\hat{u}:=\mathcal{U}-\Theta$ (see \eqref{OS1}). By Theorem \ref{cnv1} and the same bounds as those used to prove Lemma \ref{weak01} we obtain
\begin{lemma}\label{weak03}
There exists a sequence $(\alpha_i)_{i\in\mathbb{N}}$ such that $\alpha_i\to\infty$ as $i\to\infty$ and
$u_{\alpha_i}$ converges to $\hat{u}$ as $\alpha_i\to\infty$, weakly in $L^p(0,T;\mathcal{V}^p)$ and strongly in $L^p(0,T; L^p(\mathcal{H},\mu))$, $1\leq p<\infty$.

Moreover, $\frac{\partial\,}{\partial\,t}u_{\alpha_i}$ converges to $\frac{\partial\,}{\partial\,t}\hat{u}$ as $\alpha_i\to\infty$ weakly in $L^2(0,T;L^2(\mathcal{H},\mu))$.
\end{lemma}
\noindent The next Theorem generalizes Theorem \ref{infty-dim-vi} to the case of unbounded operator $A$.
\begin{theorem}\label{inf-yos-vi}
For every $1<p<\infty$ the function $\hat{u}$ is a solution of the variational problem on $\mathcal{H}$
\begin{eqnarray}\label{d-2tris}
\left\{
\begin{array}{l}
u\in L^2(0,T;\mathcal{V}^p);\:\:\:\:\frac{\partial\,u}{\partial\,t}\in L^2(0,T;L^2(\mathcal{H},\mu));\\
\\
u(T,x)=0,\:x\in\mathcal{H};\:\:\:\:u(t,x)\geq 0,\: (t,x)\in[0,T]\times\mathcal{H};\\
\\
\hspace{-3pt}\displaystyle{-\big(\frac{\partial u}{\partial t}(t),w-u(t)\big)_{\mu}+a_\mu\big(u(t),w-u(t)\big)- F\big(w-u(\,\cdot\,)\big)(t)\geq0},\\
\hspace{+190pt}\text{for a.e.~$t\in[0,T]$ and for all $w\in \mathcal{K}^p_{\mu}$.}
\end{array}\right.
\end{eqnarray}
Moreover, $\hat{u}\in C([0,T]\times\mathcal{H})$.
\end{theorem}
\noindent We omit the proof which follows from Lemma \ref{weak03}, Proposition \ref{prot}, \eqref{prot2} and it goes through arguments similar to (but simpler than) those adopted in the proof of Theorem \ref{infty-dim-vi}. Continuity of the solution is a consequence of Corollary \ref{ascoli}.

The optimal stopping time of $\mathcal{U}$ is found by probabilistic arguments as in Section \ref{refer}. For $(t,x)\in[0,T]\times\mathcal{H}$, let $\tau^\star_{\alpha}(t,x)$ be defined as in \eqref{tau} and set
\begin{equation}\label{gnom}
\tau^\star(t,x):=\inf\{s\geq t\,:\,\mathcal{U}(s,X^{t,x}_s)=\Theta(s,X^{t,x}_s)\}\wedge T.
\end{equation}
\begin{lemma}\label{tdar2}
{For each $(t,x)\in[0,T]\times\mathcal{H}$} there exists a sequence $(\alpha_j)_{j\in\mathbb{N}}$, with $\alpha_j=\alpha_j(t,x)$, such that $\alpha_j\to\infty$ as $j\to\infty$ and
\begin{equation}
\lim_{j\to\infty}(\tau^\star(t,x)\wedge\tau^\star_{\alpha_j}(t,x))(\omega)=\tau^\star(t,x)(\omega),
\quad\mathbb{P}\textrm{-a.e.}\:\omega\in\Omega.
\end{equation}
\end{lemma}
\begin{proof}
The proof follows along the same lines of that of Lemma \ref{tdar} and it is based on Corollary \ref{ascoli} and Proposition \ref{conv-yos}.
\end{proof}
\begin{theorem}\label{os-t2}
The stopping time $\tau^\star(t,x)$ is optimal for $\mathcal{U}(t,x)$.
\end{theorem}
\begin{proof}
Set $\tau^\star=\tau^\star(t,x)$ for simplicity. Take $\sigma=\tau^\star$ in \eqref{ualpha-dyn} to obtain
\begin{eqnarray}\label{sprot}
\mathcal{U}_\alpha(t,x)=\mathbb{E}\left\{\mathcal{U}_\alpha(\tau^\star\wedge\tau_\alpha^\star,
X^{(\alpha)t,x}_{\tau^\star\wedge\tau_\alpha^*})\right\}.
\end{eqnarray}
Consider the subsequence $(\mathcal{U}_{\alpha_j})_{j\in\mathbb{N}}$ corresponding to the sequence $(\alpha_j)_{j\in\mathbb{N}}$ given in Lemma \ref{tdar2}, and take limits in \eqref{sprot} as $j\to\infty$. Proposition \ref{conv-yos}, Corollary \ref{ascoli} and arguments as in the proof of Theorem \ref{os-t1} allow us to conclude that \begin{eqnarray}\label{vfun-fin}
\mathcal{U}(t,x)=\mathbb{E}\left\{\mathcal{U}(\tau^\star,X^{t,x}_{\tau^\star})\right\}=
\mathbb{E}\left\{\Theta(\tau^\star,X^{t,x}_{\tau^\star})\right\}.
\end{eqnarray}
That is, $\tau^\star$ is optimal.
\end{proof}


\section{Uniqueness in a particular case}\label{uniqueness}

{We address the question of uniqueness of the solution to problem \eqref{d-2tris}
only in the case of processes $X$ whose Kolmogorov operator generates a symmetric Ornstein-Uhlenbeck semigroup (cf.~\cite{DaPr-Zab04}, Chapters 6 and 7). For instance, Chow and Menaldi \cite{Ch-Men89} consider such dynamics while carrying out an analysis similar to ours.}

In \eqref{SDE-infty} we take $\sigma(x)\equiv 1$ and repalce  $W^0$ by a $Q$-Wiener process $(W_t)_{t\in[0,T]}$ taking values in $\mathcal{H}$ (cf.~\cite{DaPr-Zab}, Chapter 4 and Remark 5.1 of Chapter 5), with covariance operator $Q\in\mathcal{L}(\mathcal{H})$ positive and of trace-class.
We make the following assumption on $A$.
\begin{ass}\label{ass:diss}
The operator $A$ is negative, self-adjoint and there esists $m>0$ such that $\langle Ax,x\rangle_{\mathcal{H}}\le -m\|x\|^2_{\mathcal{H}}$. Moreover $Tr\big[QA^{-1}\big]_{\mathcal{H}}<+\infty$ and $e^{tA}Q=Qe^{tA}$ for all $t>0$.
\end{ass}
\noindent Then the semigroup generated by the Kolmogorov operator associated to $X$ is symmetric (cf.~\cite{DaPr-Zab04}, Corollary 10.1.7), and admits a centered Gaussian invariant measure $\nu$ (cf.~\cite{DaPr-Zab04}, Proposition 10.1.1) with covariance operator $\Gamma$ defined by
\begin{align}\label{def-gamma}
\Gamma:=-\frac{1}{2}A^{-1}Q
\end{align}
(cf.~\cite{DaPr-Zab04}, Proposition 10.1.6). For $\varphi_k$ and $\lambda_k$ as in \eqref{bari} the $Q$-Wiener process may be represented as $W_t=\sum_{k}\sqrt{\lambda_k}\beta^k_t\,\varphi_k=:Q^{\frac{1}{2}}B_t$ where $\big\{\beta^k_t,\,t\ge0,\,k\in\mathbb{N}\big\}$ is an infinite sequence of independent, real, standard Brownian motions and $B_t:=\sum_{k}\beta^k_t\,\varphi_k$.
Therefore, the SDE for $X$ may be formally written as
\begin{align}\label{formalSDE}
dX_t=AX_t\,dt+Q^{\frac{1}{2}} dB_t,\quad t\in[0,T].
\end{align}

Now the variational problem may be set in the Gauss-Sobolev space associated to the measure $\nu$ rather than that associated to $Q$. All arguments developed in the previous sections may be carried out and, in particular, Theorems \ref{inf-yos-vi} and \ref{os-t2} hold with $\mathcal{V}^p$ replaced by $W^{1,2}(\mathcal{H},\nu)$, with $a_\mu(\,\cdot\,,\,\cdot\,)$ replaced by
\begin{align}\label{chow-bil}
a_{\nu}&(u,w):=\int_{\mathcal{H}}\frac{1}{2}\langle Q\,Du(x),D w(x)\rangle _\mathcal{H}\nu(dx),\quad u,w\in W^{1,2}(\mathcal{H},\nu)
\end{align}
and with $F(\,\cdot\,)(t)$ replaced by the dual pairing 
\begin{align}\label{dualF}
\langle\langle f(t),w\rangle\rangle\hspace{-2pt}:=\hspace{-2pt}\big(\frac{\partial\,\Theta}{\partial\,t}(t),w\big)_\nu
-
a_{\nu}(\Theta(t),w)
\quad\text{for $w\in W^{1,2}(\mathcal{H},\nu)$.}
\end{align}

Notice that conditions \eqref{condth01} are sufficient to guarantee the well posedness of \eqref{dualF} and
that it is no longer needed to introduce the operator $L_A$ of Section \ref{sec-yos} and its continuous extension; also, $A\Gamma A$ is not necessarily of trace class and hence the analogue of Assumption \ref{ass-Q} in this setting (i.e.\ $Tr\big[A\Gamma A \big]_{\mathcal{H}}<+\infty$), breaks down. However, here we do not need to rely on that assumption since the existence of the Gaussian invariant measure and the particular form of its covariance operator $\Gamma$ (cf.\ \eqref{def-gamma}) substantially simplify the bilinear form.

The uniqueness in $L^2(\mathcal{H},\nu)$ of the solution of the variational inequality now follows from usual comparison arguments as in \cite{Ben-Lio82} and the fact that 
\begin{align}
a_\nu(u,u)+\eta\big(u,u\big)_\nu\ge \eta\big(u,u\big)_\nu,\quad\text{for any $\eta>0$.}
\end{align}

\begin{remark}
Notice that our approach allow to give a positive answer to the open question in Remark 2, of \cite{Ch-Men89}, p.~49, under assumptions similar to those required there, although in the finite time-horizon case. Also, it solves the problem posed in Section 5 of \cite{Ch-Men89} (see discussion following Theorem 3, p.\ 51, therein) regarding the connection between infinite dimensional variational inequalities and optimal stopping problems when $\sigma$ depends on the process. We believe that our method extends to the infinite time-horizon case under quite natural integrability assumptions.
\end{remark}

\begin{remark}\label{excess}
The above arguments suggest that when a Gaussian invariant measure can be found, then uniqueness is more likely to be obtained as well. That naturally links our work to \cite{Bar-Mar08}, \cite{Bar-Sri06} and \cite{Zab01}, where variational problems associated to optimal stopping ones are solved in Sobolev spaces with respect to excessive measures (possibly invariant) of the diffusion process' semigroup.

Our proof of existence of a solution to the variational problem and its connection to the optimal stopping one could be possibly replicated when the Gaussian measure $\mu$ is replaced by an excessive measure $\nu$ (possibly invariant) provided that derivatives of $\nu$ along the basis vectors' directions exist (in the sense of \cite{Bogachev}, Definition 5.1.3) and natural integrability conditions hold, together with some refinements of Assumptions \ref{ass-sigma} and \ref{ass-Q}. Then uniqueness of the solution of the variational problem \eqref{d-2tris} would follow as shown in \cite{Bar-Mar08}, \cite{Bar-Sri06} and \cite{Zab01}.
\end{remark}


\appendix
\section*{Appendix}\label{appendix}
\renewcommand{\theequation}{A-\arabic{equation}}

\begin{proof}[Proof of Proposition \ref{unb-d}]
Fix $(t,x^{(n)})\in[0,T]\times\mathbb{R}^n$ and take $\overline{R}>0$ such that $x^{(n)}\in\mathcal{O}_{\overline{R}}$. Now for all $R\geq \overline{R}$ we have
\begin{align*}
0\leq&\hspace{+4pt}\mathcal{U}^{(n)}_{\alpha}(t,x^{(n)})-\mathcal{U}^{(n)}_{\alpha,R}(t,x^{(n)})\\
&\leq\hspace{-4pt}\sup_{t\leq\sigma\leq T}\hspace{-4pt}\mathbb{E}\left\{\hspace{-2pt}\left(\Theta^{(n)}(\sigma,X^{(\alpha)t,x;n}_{\sigma})-\Theta^{(n)}
(\tau_{R},
X^{(\alpha)t,x;n}_{\tau_{R}})\right)I_{\{
\sigma>\tau_{R}\}}\hspace{-2pt}\right\}\le 2\,\overline{\Theta}\,\mathbb{P}\big(\tau_R<T\big),
\end{align*}
by \eqref{psi1} and with $I_{\{\sigma>\tau_{R}\}}$ the indicator function of the set $\{\sigma>\tau_{R}\}$. By Markov inequality and standard estimates for strong solutions of SDEs in $\mathbb{R}^n$ (cf.~for instance \cite{Krylov} Chapter 2, Section 5, Corollary 12), it follows
\begin{align*}
\mathbb{P}\big(\tau_R<T\big)&\le \mathbb{P}\Big(\sup_{t\le s\le T}\big\|X^{(\alpha)t,x;n}_s-x^{(n)}\big\|_{\mathbb{R}^n}>R-\overline{R}\Big)\\
&\le \frac{\mathbb{E}\Big\{\sup_{t\le s\le T}\big\|X^{(\alpha)t,x;n}_s-x^{(n)}\big\|_{\mathbb{R}^n}\Big\}}{(R-\overline{R})}\le C_{n,\alpha,T}(1+\big\|x^{(n)}\big\|_{\mathbb{R}^n})\frac{(T-t)^{\frac{1}{2}}}{(R-\overline{R})}
\end{align*}
with $C_{n,\alpha,T}>0$, only depending on $(\alpha,n,T)$ and bounds on $\sigma$.

Therefore
\begin{eqnarray*}
\lim_{R\to\infty}\sup_{(t,x^{(n)})\in[0,T]\times\mathcal{K}}\big|\mathcal{U}^{(n)}_{\alpha,R}(t,x^{(n)})-\mathcal{U}^{(n)}_{\alpha}(t,x^{(n)})\big|=0
\end{eqnarray*}
for every compact subset $\mathcal{K}\subset\mathbb{R}^n$. If all $\mathcal{U}^{(n)}_{\alpha,R}$, are continuous, then $\mathcal{U}^{(n)}_{\alpha}$ is continuous on every compact subset $[0,T]\times\mathcal{K}$ and this is enough for global continuity in $\mathbb{R}^n$.
\end{proof}

\begin{proof}[Proof of Corollary \ref{str-sol-obs}]
By the regularity of $\bar{u}$ in Corollary \ref{bl-str}, it is well known that the expression \eqref{str-f} is equivalent to
\begin{eqnarray*}
\max\left\{\frac{\partial \bar{u}}{\partial t}+\mathcal{L}_{\alpha,n}\bar{u}+f_{\alpha,n}\,,\,-\bar{u}\right\}=0,
\qquad\textrm{a.e.}\,\in[0,T]\times\overline{\mathcal{O}}_R.
\end{eqnarray*}
(see for instance \cite{Ben-Lio82}, Chapter 3, Section 1, p.~191).

The regularity of $\partial\mathcal{O}_R$ and \cite{Adams}, Theorem 3.22 enable us to find a sequence $(u_j)_{j\in\mathbb{N}}$, such that $u_j\in C^\infty_c(\mathbb{R}^{n+1})$ and
\begin{eqnarray}\label{conv-val}
\|u_j-\bar{u}\|_{W^{1\,2,p}((0,T)\times\mathcal{O}_R)}\to 0\qquad\textrm{as}\:j\to\infty.
\end{eqnarray}
In fact it suffices to take a partition of the domain and use the standard mollification on each element of the partition. Then \eqref{conv-val} follows from the usual properties of the mollifiers and the fact that the operators $\partial_t$, $D$ and $D^2$ are closed in $L^p$. Moreover, the continuity of $\bar{u}$ and that of a suitable extension to $\mathbb{R}^{n+1}$ imply that the convergence is also uniform on any compact set $\mathcal{O}^\prime$ such that $[0,T]\times\overline{\mathcal{O}}_R\subset\mathcal{O}^\prime$; that is
\begin{eqnarray}\label{conv-val2}
\|u_j-\bar{u}\|_{L^\infty}\to 0,\qquad\:\textrm{as}\:j\to\infty,\:\: \textrm{on}\:\mathcal{O}^{\prime}.
\end{eqnarray}

Now we fix an arbitrary $t\in[0,T]$ and a stopping time $\tau\in[t,T]$. An application of Dynkin's formula from $t$ to $\tau\wedge\tau_R$ gives
\begin{eqnarray}\label{dyn-prg}
\mathbb{E}\left\{u_j(\tau\wedge\tau_{R},X^{(\alpha)t,x;n}_{\tau\wedge\tau_{R}})\right\}=u_j(t,x^{(n)})+\mathbb{E}
\left\{\int_t^{\tau\wedge\tau_{R}}{\left(\frac{\partial u_j}{\partial s}+\mathcal{L}_{\alpha,n}u_j\right)(s,X^{(\alpha)t,x;n}_s)ds}\right\}.
\end{eqnarray}
On the other hand by \cite{Ben-Lio82}, Chapter 2, Lemma 8.1 there exists a constant $C_{T,R}>0$ such that
\begin{align}
\hspace{-16pt}\left|\mathbb{E}\left\{\int_t^{\tau\wedge\tau_{R}}{\hspace{-5pt}\left(\frac{\partial}{\partial s}+\mathcal{L}_{\alpha,n}\right)\hspace{-3pt}\left(u_j-\bar{u}\right)(s,X^{(\alpha)t,x;n}_s)ds}\right\}\right|
\hspace{-2pt}\leq\hspace{-2pt} C_{T,R}\left\|\left(\frac{\partial}{\partial s}+\mathcal{L}_{\alpha,n}\right)\hspace{-3pt}\left(u_j-\bar{u}\right)\right\|_{L^2((0,T)\times\mathcal{O}_R)}\hspace{-4pt},
\end{align}
hence by taking the limit as $j\to\infty$ and by using (\ref{conv-val}) and (\ref{conv-val2}) we obtain
\begin{align}\label{dyn-prg-2}
\mathbb{E}\left\{\hspace{-2pt}\bar{u}(\tau\wedge\tau_{R},X^{(\alpha)t,x;n}_{\tau\wedge\tau_{R}})\hspace{-2pt}\right\}
\hspace{-2pt}=\hspace{-2pt}\bar{u}(t,x^{(n)})+\mathbb{E}
\left\{\hspace{-2pt}\int_t^{\tau\wedge\tau_{R}}{\hspace{-4pt}\left(\frac{\partial \bar{u}}{\partial s}+\mathcal{L}_{\alpha,n}\bar{u}\right)\hspace{-2pt}(s,X^{(\alpha)t,x;n}_s)ds}\hspace{-2pt}\right\}\:\textrm{for all}\:\tau\in[t,T].
\end{align}
Recall that \eqref{obs-1} holds almost everywhere in $(0,T)\times\mathcal{O}_R$ and, being the diffusion uniformly non degenerate, the law of $X^{(\alpha)t,x;n}$ is absolutely continuous with respect to the Lebesgue measure on $(0,T)\times\mathcal{O}_R$. Then
\begin{eqnarray}\label{druid1}
\bar{u}(t,x^{(n)})\geq\mathbb{E}\left\{\int_t^{\tau\wedge\tau_R}{f_{\alpha,n}(s,X^{(\alpha)t,x;n}_s)ds}\right\}\qquad\textrm{for all}\:\tau\in[t,T];
\end{eqnarray}
in particular, with $\tau^\star$ defined by
\begin{equation}\label{druid2}
\tau^\star:=\inf\{s\geq t\,:\,\bar{u}(s,X^{(\alpha)t,x;n}_s)=0\}\wedge{{\tau_R}}\wedge T,
\end{equation}
\eqref{obs-1} implies
\begin{eqnarray}\label{druid3}
\bar{u}(t,x^{(n)})=\mathbb{E}\left\{\int_t^{{{\tau^\star}}}{f_{\alpha,n}(s,X^{(\alpha)t,x;n}_s)ds}\right\}.
\end{eqnarray}
Therefore, by using \eqref{gen-inf} and by recalling \eqref{OS4} we have
\begin{align}\label{druid4}
\bar{u}(t,x^{(n)})&=\sup_{t\leq\tau\leq T}\mathbb{E}\left\{\int_t^{\tau\wedge\tau_R}{\left(\frac{\partial\Theta^{(n)}}{\partial s}+\mathcal{L}_{\alpha,n}\Theta^{(n)}\right)(s,X^{(\alpha)t,x;n}_s)ds}\right\}\nonumber\\
&=\sup_{t\leq\tau\leq T}\mathbb{E}\left\{\Theta^{(n)}(\tau\wedge\tau_R,X^{(\alpha)t,x;n}_{\tau\wedge\tau_R})\right\}-\Theta^{(n)}(t,x)
\hspace{+2pt}=\hspace{+3pt}\mathcal{U}^{(n)}_{\alpha,R}(t,x)-\Theta^{(n)}(t,x).
\end{align}
It now follows that $\bar{u}=u^{(n)}_{\alpha,R}$ and $\tau^\star=\tau^\star_{\alpha,n,R}$.

Notice that for any stopping time $\tau\leq\tau^\star_{\alpha,n,R}$, combining \eqref{druid4} and \eqref{dyn-prg-2} gives
\begin{eqnarray}\label{druid5}
\mathcal{U}^{(n)}_{\alpha,R}(t,x^{(n)})=\mathbb{E}\left\{\mathcal{U}^{(n)}_{\alpha,R}({{\tau}},X^{(\alpha)t,x;n}
_{{{\tau}}})\right\},
\end{eqnarray}
i.e. the dynamic programming principle for $\mathcal{U}^{(n)}_{\alpha,R}$ holds.
\end{proof}
\begin{proof}[Proof of Lemma \ref{bound-f}]
Set $u^R:=u^{(n)}_{\alpha,R}$ and recall Corollary \ref{str-sol-obs}. An application of It\^o's formula based on the same arguments as those that lead to \eqref{dyn-prg-2} gives
\begin{align}\label{app-boundf00}
\mathbb{E}&\left\{e^{-\int^{\tau^x_R}_t{\,\frac{1}{\varepsilon}\nu(s)ds}}\:u^R\big
(\tau^x_R,X^{(\alpha)t,x;n}_{\tau^x_R}\big)-
e^{-\int^{\tau^x_R\wedge\tau^y_R}_t{\,\frac{1}{\varepsilon}\nu(s)ds}}\:u^R\big
(\tau^x_R\wedge\tau^y_R,X^{(\alpha)t,x;n}_{\tau^x_R\wedge\tau^y_R}\big)\right\}\nonumber\\
\le& -\mathbb{E}\left\{\int^{\tau^x_R}_{\tau^x_R\wedge\tau^y_R}{
e^{-\int^s_t{\frac{1}{\varepsilon}\nu(u)du}}f_{\alpha,n}\big(s,X^{(\alpha)t,x;n}_s\big)ds}\right\}.
\end{align}
For the left-hand side of \eqref{app-boundf00} we observe that on the set $\big\{\tau^x_R\le\tau^y_R\big\}$ the difference inside the expectation is zero, whereas on the set $\big\{\tau^x_R>\tau^y_R\big\}$ one has
\begin{align}\label{app-boundf01}
u^R\big
(\tau^x_R,X^{(\alpha)t,x;n}_{\tau^x_R}\big)=0=u^R\big
(\tau^x_R\wedge\tau^y_R,X^{(\alpha)t,y;n}_{\tau^x_R\wedge\tau^y_R}\big)\qquad\text{$\mathbb{P}$-a.s.}
\end{align}
Therefore from \eqref{app-boundf00}, \eqref{reps1}, \eqref{psi2}, \eqref{lip-V} and Lemma \ref{lipSDE} we obtain
\begin{align}\label{app-boundf02}
\mathbb{E}\bigg\{\int^{\tau^x_R}_{\tau^x_R\wedge\tau^y_R}&{
e^{-\int^s_t{\frac{1}{\varepsilon}\nu(u)du}}f_{\alpha,n}\big(s,X^{(\alpha)t,x;n}_s\big)ds}\bigg\}\nonumber\\
\le&\mathbb{E}\left\{\big|u^R\big
(\tau^x_R\wedge\tau^y_R,X^{(\alpha)t,y;n}_{\tau^x_R\wedge\tau^y_R}\big)-
u^R\big
(\tau^x_R\wedge\tau^y_R,X^{(\alpha)t,x;n}_{\tau^x_R\wedge\tau^y_R}\big)\big|\right\}\\
\le& \big(L_\Theta+L_\mathcal{U}\big)C_{1,T}\big\|x^{(n)}-y^{(n)}\big\|_{\mathcal{H}}\,.\nonumber
\end{align}

To obtain \eqref{bound-f00} we need to find a similar bound for the first member of \eqref{app-boundf02} but from below. For that we introduce the auxiliary problem
\begin{align}\label{app-boundf03}
v^R(t,x^{(n)}):=\inf_{t\le\tau\le T}\mathbb{E}\left\{\int_t^{\tau\wedge\tau_R}{f_{\alpha,n}(s,X^{(\alpha)t,x;n}_s)ds}\right\}\qquad\text{for $(t,x^{(n)})\in[0,T]\times\mathbb{R}^n$}
\end{align}
and we observe that same arguments as those used to obtain Proposition \ref{bl-str} and Corollary \ref{str-sol-obs} give $v^R\in L^p(0,T;W^{1,p}_0(\mathcal{O}_R))\cap L^p(0,T;W^{2,p}(\mathcal{O}_R))$ and $\frac{\partial\,v^R}{\partial\,t}\in L^p(0,T;L^p(\mathcal{O}_R))$, for all $1\le p<+\infty$. Moreover $v^R$ uniquely solves, in the almost everywhere sense, the obstacle problem
\begin{eqnarray}\label{obs-2}
\left\{
\begin{array}{ll}
\max\left\{-\displaystyle{\frac{\partial v}{\partial t}}-\mathcal{L}_{\alpha,n}v-f_{\alpha,n}\,,\,v\right\}(t,x^{(n)})= 0,\:\:\:\: (t,x^{(n)})\in(0,T)\times\mathcal{O}_R, &\\
\\
v(t,x^{(n)})\leq0\:\:\textrm{on}\:\:[0,T]\times\overline{\mathcal{O}}_R;\:\: v(T,x^{(n)})=0,\:\:\:\: x^{(n)}\in\overline{\mathcal{O}}_R.&\\
\end{array}
\right.
\end{eqnarray}
Again, by arguing as above for \eqref{app-boundf00} and by replacing $u^R$ by $v^R$, the reversed inequality is obtained. Hence, the analogous for $v^R$ of \eqref{app-boundf01} gives
\begin{align}\label{app-boundf04}
\mathbb{E}\bigg\{\int^{\tau^x_R}_{\tau^x_R\wedge\tau^y_R}{
e^{-\int^s_t{\frac{1}{\varepsilon}\nu(u)du}}f_{\alpha,n}\big(s,X^{(\alpha)t,x;n}_s\big)ds}\bigg\}\ge -\big(L_\Theta+L_\mathcal{U}\big)C_{1,T}\big\|x^{(n)}-y^{(n)}\big\|_{\mathcal{H}}\,.
\end{align}
Now \eqref{bound-f00} follows by \eqref{app-boundf02} and \eqref{app-boundf04}.
\end{proof}

\begin{proof}[Proof of Lemma \ref{lip-ue}]
It is enough to show that $\|u^R_\varepsilon(t,x^{(n)})-u^R_\varepsilon(t,y^{(n)})\|\le L_P\|x^{(n)}-y^{(n)}\|_{\mathcal{H}}$ for all $t\in[0,T]$ and $x,y\in\mathcal{H}$. Recalling \eqref{gen-inf} and \eqref{penal01}, we find
\begin{align}\label{lip-ue01}
u^R_\varepsilon(t,x^{(n)})&-u^R_\varepsilon(t,y^{(n)})
\nonumber\\
\le&\big|\Theta^{(n)}(t,x^{(n)})-\Theta^{(n)}(t,y^{(n)})\big|\nonumber\\
&+
\sup_\nu\inf_{\nu'}\mathbb{E}\bigg\{\int^{\tau^x_R}_t{e^{-\int^s_t{\frac{1}{\varepsilon}\nu(u)du}}\frac{1}
{\varepsilon}\nu(s)\Theta^{(n)}\big(s,X^{(\alpha)t,x;n}_s\big)ds}\nonumber\\
&\hspace{+60pt}
-\int^{\tau^y_R}_t{e^{-\int^s_t{\frac{1}
{\varepsilon}\nu'(u)du}}\frac{1}
{\varepsilon}\nu'(s)\Theta^{(n)}\big(s,X^{(\alpha)t,y;n}_s\big)ds}\\
&\hspace{+100pt}+e^{-\int^{\tau^x_R}_t{\frac{1}{\varepsilon}\nu(s)}ds}
\Theta^{(n)}\big(\tau^x_R,X^{(\alpha)t,x;n}_{\tau^x_R}\big)
\nonumber\\
&\hspace{+140pt}-e^{-\int^{\tau^y_R}_t{\frac{1}{\varepsilon}\nu'(s)}ds}
\Theta^{(n)}\big(\tau^y_R,X^{(\alpha)t,y;n}_{\tau^y_R}\big)\bigg\}.\nonumber
\end{align}
From It\^o's formula, \eqref{gen-inf} and Lemma \ref{bound-f} one finds
\begin{align}\label{lip-ue02}
\mathbb{E}\bigg\{e^{-\int^{\tau^x_R}_t{\frac{1}{\varepsilon}\nu(s)}ds}\,&
\Theta^{(n)}\big(\tau^x_R,X^{(\alpha)t,x;n}_{\tau^x_R}\big)\bigg\}\\
\le&L_f\big\|x^{(n)}-y^{(n)}\big\|+\mathbb{E}\bigg\{e^{-\int^{\tau^x_R\wedge\tau^y_R}_t{\frac{1}{\varepsilon}
\nu(s)}ds}\,
\Theta^{(n)}\big(\tau^x_R\wedge\tau^y_R,X^{(\alpha)t,x;n}_{\tau^x_R\wedge\tau^y_R}\big)\bigg\}\nonumber\\
&-\mathbb{E}\bigg\{\int^{\tau^x_R}_{\tau^x_R\wedge\tau^y_R}{e^{-\int^s_t{\frac{1}{\varepsilon}\nu(u)du}}
\frac{1}
{\varepsilon}\nu(s)\Theta^{(n)}\big(s,X^{(\alpha)t,x;n}_s\big)ds}\bigg\}\nonumber
\end{align}
and similarly,
\begin{align}\label{lip-ue03}
\mathbb{E}\bigg\{e^{-\int^{\tau^y_R}_t{\frac{1}{\varepsilon}\nu'(s)}ds}\,&
\Theta^{(n)}\big(\tau^y_R,X^{(\alpha)t,y;n}_{\tau^y_R}\big)\bigg\}\\
\ge&-L_f\big\|x^{(n)}-y^{(n)}\big\|+\mathbb{E}\bigg\{e^{-\int^{\tau^x_R\wedge\tau^y_R}_t{\frac{1}{\varepsilon}
\nu'(s)}ds}
\,\Theta^{(n)}\big(\tau^x_R\wedge\tau^y_R,X^{(\alpha)t,y;n}_{\tau^x_R\wedge\tau^y_R}\big)\bigg\}\nonumber\\
&-\mathbb{E}\bigg\{\int^{\tau^y_R}_{\tau^x_R\wedge\tau^y_R}{e^{-\int^s_t{\frac{1}{\varepsilon}\nu'(u)du}}
\frac{1}
{\varepsilon}\nu'(s)\Theta^{(n)}\big(s,X^{(\alpha)t,y;n}_s\big)ds}\bigg\}.\nonumber
\end{align}
Take now
\begin{align}\label{nuprime}
\text{$\nu'(s)=\nu(s)$ for $s\in(t,\tau^x_R\wedge\tau^y_R]$\: and\: $\nu'(s)=0$ for $s>\tau^x_R\wedge\tau^y_R$,}
\end{align}
then from \eqref{lip-ue01}, \eqref{nuprime}, \eqref{lip-ue02}, \eqref{lip-ue03} and recalling \eqref{psi2} and Lemma \ref{lipSDE} we obtain
\begin{align}\label{lip-ue04}
u^R_\varepsilon(t,x^{(n)})&-u^R_\varepsilon(t,y^{(n)})
\nonumber\\
\le& \big(2L_f+L_\Theta\big)\big\|x^{(n)}-y^{(n)}\big\|_{\mathcal{H}}\nonumber\\
&+
\mathbb{E}\left\{\Big|\Theta^{(n)}\big(\tau^x_R\wedge\tau^y_R,X^{(\alpha)t,x;n}_{\tau^x_R\wedge\tau^x_R}
\big)-
\Theta^{(n)}\big(\tau^y_R\wedge\tau^x_R,X^{(\alpha)t,y;n}_{\tau^y_R\wedge\tau^x_R}\big)\Big|\right\}\\
&+
\sup_\nu\mathbb{E}\bigg\{\int^{\tau^x_R}_{\tau^x_R\wedge\tau^y_R}
{e^{-\int^s_t{\frac{1}{\varepsilon}\nu(u)du}}\frac{1}
{\varepsilon}\nu(s)\big(\Theta^{(n)}\big(s,X^{(\alpha)t,x;n}_s\big)-\Theta^{(n)}\big(s,X^{(\alpha)t,y;n}_s\big)\big)ds}
\bigg\}\nonumber\\
\le& \big(2L_f+L_\Theta+2L_\Theta C_{1,T}\big)\big\|x^{(n)}-y^{(n)}\big\|_{\mathcal{H}}.\nonumber
\end{align}
One can argue in a similar way to bound $u^R_\varepsilon(t,y^{(n)})-u^R_\varepsilon(t,x^{(n)})$.
\end{proof}

\vspace{+8pt}
\ackn{During this work the second named author was funded by the University of Rome ``La Sapienza'' through the PhD programme in \emph{Mathematics for Economic-Financial Applications} and by the EPSRC grant EP/K00557X/1}

\end{document}